\documentclass[preprint ,11pt,french,english]{article}
\usepackage{amsmath,amssymb,amsthm} 
\usepackage{amstext,amsfonts,mathtools}
\usepackage{mathrsfs,bm,bbm,enumerate}
\usepackage{graphicx,color,tikz}
\usepackage{caption,subcaption,geometry}
\usepackage{babel}
\usepackage[T1]{fontenc} 
\usepackage[latin9]{inputenc}
\usepackage{soul}
\usepackage{dsfont}
 \usepackage{hyperref}
\usepackage{graphicx,epsfig}
\usepackage{epstopdf}
\usepackage{ifthen}
\graphicspath{{eps/}}
\newtheorem{proposition}{\textbf{Proposition}}
\newtheorem{theorem}{\textbf{Theorem}}

\newtheorem{definition}{\textbf{Definition}}

\providecommand{\review}[1]{#1}
\providecommand{\revise}[1]{{#1}}

\newcommand{\inR}{\in \mathbb{R}}

\newcommand{\R}{ \mathbb{R}}
\newcommand{\Z}{ \mathbb{Z}}
\newcommand{\N}{ \mathbb{N}}

\newcommand{\Lop}{{\rm L}}
\newcommand{\Dop}{{\rm D}}

\newcommand{\dint}{{\rm d}}

\newcommand{\bx}{{\boldsymbol x}}
\newcommand{\bw}{{\boldsymbol \omega}}

\DeclareMathOperator*{\esssup}{ess\,sup}

\def\V#1{{\boldsymbol{#1}}}         
\def\Spc#1{{\mathcal{#1}}}  
\def\M#1{{\bf{#1}}}  
\def\Op#1{{\mathrm{#1}}}  
\def\ee{\mathrm{e}} 
\def\jj{\mathrm{j}} 
 
\def\One{\mathbbm{1}}

\title{Splines are universal solutions of linear inverse problems with generalized-TV regularization}

\author{Michael Unser\thanks{Biomedical Imaging Group, \'Ecole polytechnique f\'ed\'erale de Lausanne (EPFL), CH-1015 Lausanne, Switzerland}
\and
Julien Fageot\footnotemark[1]
\and
John Paul Ward\footnotemark[1]
\thanks{Department of Mathematics, University of Central Florida, Orlando, FL, USA. Present address: Department of Mathematics, North Carolina A\&T State University, Greensboro, NC, USA}
}

\begin{document}

\maketitle

\begin{abstract}
Splines come in a variety of flavors that can be characterized in terms of some differential operator $\Lop$. The simplest piecewise-constant model corresponds to the derivative operator. Likewise, one can extend the traditional notion of total variation by considering more general operators than the derivative. 
\revise{This results in the definition of a generalized total variation semi-norm and of its corresponding native space,}
which is further identified as the direct sum of two Banach spaces.
%
We then prove that the minimization of the generalized total variation (gTV), 
subject to some arbitrary (convex) consistency constraints on the linear measurements of the signal, admits nonuniform $\Lop$-spline solutions with fewer knots than the number of measurements. This shows that nonuniform splines are universal solutions of continuous-domain linear inverse problems with LASSO, $L_1$, or total-variation-like regularization constraints. Remarkably, the type of spline is fully determined by the choice of $\Lop$ and does not depend on 
the actual nature of the measurements.

\end{abstract}

\textbf{Keywords.}  
  Sparsity, total variation, splines, inverse problems, compressed sensing

\textbf{AMS subject classifications.}
  41A15, 
  47A52, 
  94A20, 
  46E27, 
  46N20, 
  47F05, 
  34A08, 
  26A33 



\section{Introduction}
Imposing sparsity constraints is a powerful paradigm for solving ill-posed inverse problems and/or for reconstructing signals at super-resolution \cite{Bruckstein2009}. 
\review{This is usually achieved by formulating the 
task as an optimization problem that includes some form of $\ell_1$ regularization \cite{Tibshirani1996}.}
The concept is central to the theory of compressed sensing (CS) \cite{Candes2007,Donoho2006} and is currently driving the development of a new generation of algorithms for the reconstruction of biomedical images \cite{Lustig2007}.
\review{The primary factors that have contributed to making sparsity a remarkably popular research topic during the past decade are as follows:
\begin{itemize}
\item the possibility of recovering the signal from few measurements (CS) with a theoretical guarantee of perfect recovery under strict conditions \cite{Candes2008a,Candes2006, Donoho2006};
\item  the availability of fast iterative solvers for this class of problems \cite{Beck2009b, Daubechies2004,Figueiredo2003, Ramani2011};
\item the increasing evidence of the superiority of the sparsity-promoting schemes over the classical linear reconstruction
(including the Tikhonov $\ell_2$ regularization) in a variety of imaging modalities  \cite{Lustig2007}.
\end{itemize}
}

The approach developed in this paper is also driven by the idea of sparsity. However, it deviates from the standard paradigm because the
recovery problem is formulated in the continuous domain under the practical constraint of a finite number of linear measurements. The ill-posedness of the problem is then dealt with by searching for a solution that is consistent with the measurements and that minimizes a generalized version of the total-variation (TV) semi-norm---the continuous-domain counterpart of $\ell_1$ regularization. Our major finding (Theorem 1) is that the extremal points of this kind of recovery problem are nonuniform splines whose type is matched to the regularization operator $\Lop$. The powerful aspect is that the result holds in full generality, as long as the problem remains convex. The only constraint is that the linear inverse problem should be well-posed over the (very small) null space of the regularization operator, which is the minimal requirement for any valid regularization scheme. \revise{In particular, Theorem 1 gives a theoretical explanation of the well-documented observation that total variation regularization---the simplest case of the present theory with $\Lop=\Dop$ (derivative operator)---tends to produce piecewise-constant solutions \cite{Chambolle2004, Rudin1992}. Recognizing the intimate connection between linear inverse problems and splines is also helpful for discretization purposes because it provides us with a parametric representation of the solution that is controlled by the regularization operator $\Lop$.}
In that respect, our representer theorems extend some older results on spline interpolation with minimum $L_1$-norms, including the adaptive regression splines of Mammen and van de Geer \cite{Mammen1997}
and the functional analytic characterization of Fisher and Jerome \cite{Fisher1975}. There is a connection as well with the work of Steidl {\em et~al.}~on splines and higher-order TV \cite{steidl2006splines}, although their formulation is strictly discrete and restricted to the denoising problem.

\review{\section{Linear Inverse Problems: Current Status and Motivation} \revise{Our notational convention is to use bold letters to denote ordinary vectors and matrices to distinguish them from their 
infinite-dimensional counterparts; that is, functions (such as $s$) 
and linear operators (such as $\Lop$).}
Simply stated, the inverse problem is to recover a signal $s$ from a finite set of linear measurements \revise{$\M y=\M y_0(s)+ \M n \in \R^M$ where $\M n$ is a disturbance term that is usually assumed to be small and independent of $s$}. 
In most real-world problem
the unknown signal lives in the continuum so that it is appropriate to view it as an element of some Banach space $\Spc B$. Then, by the assumption of linearity, there exists a set of functionals $\nu_m \in \Spc B'$ \revise{(the continuous dual of $\Spc B$)} with $m=1,\dots,M$ such that the noise-free measurements are given by $\M y_0=\V \nu(s)=(\langle \nu_1,s\rangle, \dots,\revise{\langle \nu_{M},s\rangle})$. 
The measurement functionals $\nu_m$ are governed by the underlying physics (forward model) and assumed to be known.
Since the signal $s \in \Spc B$ is an infinite-dimensional entity and the number of measurements is finite, the inverse problem is obviously ill-posed, not to mention the fact that the true measurements $\M y$ are typically only approximate versions of $\M y_0$ since they are corrupted by noise.}
%

\review{\subsection{Finite-Dimensional Formulation} The standard approach for the resolution of such inverse problems is to select some finite-dimensional reconstruction space $\Spc V={{\rm span}\{\varphi_n\}_{n=1}^N}$ $ {\subset \Spc B}$. Based on the (simplifying) assumption that $s\in \Spc V$, one then converts the original noise-free forward model into the discretized version $\M y_0=\M A \M x$, where $\M x \inR^N$ represents the expansion coefficients of $s$ in the basis $\{\varphi_n\}_{n=1}^N$. Here, $\M A$
is the so-called {\em sensing matrix} of size  $(M \times N)$ whose entries are given by $[\M A]_{m,n}=\langle \nu_m,\varphi_n\rangle$. }

\review{The basic assumption made by the theory of compressed sensing is that there exists a finite-dimensional basis (or dictionary) $\{\varphi_n\}_{n=1}^N$ that ``sparsifies'' the class of desired signals with the property that $\|\M x\|_0\le K_0$ for some fixed $K_0$ which is (much) smaller than $N$; in other words, it should be possible to {\em synthesize} the signal exactly by restricting the expansion to no more than $K_0$ atoms in the basis $\{\varphi_n\}_{n=1}^N$  \cite{Donoho2003,Elad2010b,
Rauhut2008}.
The signal recovery is then recast as the constrained optimization problem
\begin{align}
\label{Eq:CS1}
\arg \min_{\M x \in \R^{N}}  \|\M x\|_1 \quad\mbox{ s.t. }\quad \|\M y-\M A \M x\|_2^2 \le \epsilon^2,
\end{align}
where the minimization of the $\ell_1$ norm promotes sparse solutions \cite{Tibshirani1996}. The role of the right-hand-side inequality is to encourage consistency between the noisy measurements $\M y$ and their noise-free restitution $\M y_0=\M A \revise{\M x}$. 
The popularity of \eqref{Eq:CS1}
stems from the fact that the theory of CS guarantees a faithful signal recovery from $M>2 K_0$ measurements under strict conditions on $\M A$ ({\em i.e.}, restricted isometry) \cite{Candes2008a, Candes2006, Donoho2006}.}
%

\review{
Instead of basing the recovery on the synthesis formula $\revise{s=\sum_{n} x_n \varphi_n} \in \Spc V$, one can adopt an alternative {\em analysis} or {\em regularization} point of view.
To that end, one typically assumes that $s$ is discretized in some implicit ``pixel'' basis with expansion coefficients $\M s=(s_1,\dots,s_N) \in \R^N$, where the $s_n$ are the samples of the underlying signal. The corresponding system matrix (forward model) is denoted by $\M H: \R^N \to \R^M$.
Given some appropriate regularization operator $\M L: \R^{N} \to \R^{N'}$, the idea then is to exploit the property that the transformed version of the signal, $\M L \M s$, is sparse. This translates into the optimization problem
\begin{align}
\label{Eq:anal}
\arg \min_{\M s \in \R^{N}} \|\M L \M s\|_1\quad  \mbox{ s.t. } \quad \|\M y-\M H \M s\|_2^2 \le \epsilon^2,
\end{align}
which is slightly more involved than \eqref{Eq:CS1}.  The two forms are equivalent only when $N'=N$ and $\M L$ is invertible, the connection being $\M A=\M H \M L^{-1}$. 
For computational purposes,
\eqref{Eq:anal} is often converted into the equivalent unconstrained version of the problem
\begin{align}
\label{Eq:penalized}
\arg \min_{\M s \in \R^{N}} \big(\|\M y-\M H \M s\|_2^2 + \lambda \|\M L \M s\|_1\big),
\end{align}
where $\lambda\in \R^+$ is an adjustable regularization parameter that needs to tuned such that $\|\M y-\M H \M s\|_2^2=\epsilon^2$.
One of the preferred choices for $\M L$ is the finite-difference operator---or the discrete version of the gradient in dimensions higher than one. This corresponds to the  ``total-variation'' reconstruction method, which is widely used in applications  \cite{Beck2009, Chambolle2004, Goldstein2009, Rudin1992}.
}

\review{
The sparsity-promoting effect of these discrete formulations and the conditions under which the expansion coefficients of the signal can be recovered are fairly well understood \cite{Foucart2013, Unser2016}. What is less satisfactory 
is the intrinsic interdependence between the sparsity constraints and the choice of the appropriate reconstruction space, which makes it difficult to deduce rates of convergence and error estimates relating to the underlying continuous-domain recovery problem.
}

\review{
\subsection{Infinite-Dimensional Formulation} 
\label{Sec:linkTV}
Recently, Adcock and Hansen have addressed the above limitation by formulating an infinite-dimensional theory of CS \cite{Adcock2011}.
The measurements are the same as before, but the unknown signal is now a function $s: \R^d \to \R$.
For the purpose of illustration, we take $d=1$ and $s \in {\rm BV}(\R)$ with the property that such a function admits the (unique)  expansion $\revise{s}=\sum_{n} w_n \psi_n$ in the (properly normalized) Haar
wavelet basis $\{\psi_n\}$. It is known that the condition $\revise{s} \in {\rm BV}(\R)$ implies  \revise{ the inclusion of $w=(w_n)$ in weak-$\ell_1(\Z)$---a space that is slightly larger than $\ell_1(\Z)$} \cite{Cohen2003}. Conversely, one can force the inclusion in ${\rm BV}(\R)$ by imposing a bound on the $\ell_1$ norm of these coefficients. If one further assumes that the signal is sparse in the Haar basis, one can recast the reconstruction problem as 
\begin{align}
\label{Eq:wav}
\min_{w \in \ell_1(\Z)}  \|w\|_1 \quad\mbox{ s.t. }\quad \sum_{m=1}^M \big|y_m-\langle \nu_m, \sum_{n} w_n \psi_n\rangle \big|^2 \le \epsilon^2,
\end{align}
which is the infinite-dimensional counterpart of the synthesis formulation \eqref{Eq:CS1}. The key question is to derive conditions on how to choose the $\nu_m$ to guarantee recovery of wavelet coefficients up to a certain scale. This has been done in \cite{Adcock2011, Adcock2013}, which means that the issue of convergence is now reasonably well understood for the synthesis form of the recovery problem.
}

\review{
In our framework, $\Spc M_\Dop(\R)$ is the space of functions on $\R$ of bounded (total) variation, which is slightly larger than ${\rm BV}(\R)$ because it also includes constant signals. This allows us to close the circle
by enforcing a regularization on the ``true'' total variation of the solution, \revise{which is associated with the derivative operator
$\Dop=\frac{\dint }{\dint x}$}. 
This results in the functional optimization problem
\begin{align}
\revise{s}=\arg \min_{f \in \Spc M_\Dop(\R)}  {\rm TV}(f)=\|\Dop f\|_{\Spc M} \quad\mbox{ s.t. }\quad \revise{\|\M y-\V \nu(f)\|_2^2=}\sum_{m=1}^M \big|y_m-\langle \nu_m, f \rangle \big|^2 \le \epsilon^2,
\label{eq:TVprob}
\end{align}
which is the continuous-domain counterpart of \eqref{Eq:anal}. 
Now, the motivation for our present theory is that \eqref{eq:TVprob} corresponds to
a special case of Theorem \ref{Theo:L1spline} with $\Lop=\Dop$ and the closed compact convex set $\Spc C\subset \R^M$
being specified by the
inequality on the right-hand-side of \eqref{eq:TVprob}; that is, $\Spc C_{\epsilon}(\M y) =\{\M z \in \R^M: \revise{\|\M y-\M z\|_2^2}\le \epsilon^2\} $.
The key is that the differentiation operator $\Dop$ is spline-admissible in the sense of Definition \ref{Def:splineadmis}:
Its causal Green's function is
the Heaviside (or unit-step) function $\rho_\Dop(x)=\One_+(x)$ whose rate of growth is $n_0=0$, while its null space
$\Spc N_\Dop={\rm span}\{p_1\}$ with $p_1(x)=1$ is 
composed of all constant-valued signals. This implies that the extreme points of
\eqref{eq:TVprob} necessarily take the form
\begin{align}
\label{Eq:Dspline}
s(x)=b_1  + \sum_{k =1}^K a_k\One_+(x-x_k)
\end{align}
with $K\le M$. This corresponds to a piecewise-constant signal with jumps of size $a_k$ at the $x_k$, as illustrated in Figure 1. The solution also happens to be a polynomial spline of degree $0$ with knots at the $x_k$ with the property that $\Dop\{s\}=\sum_{k=1}^K
a_k\delta(\cdot-x_k)=w_\delta$, which is a weighted sum of shifted Dirac impulses (the innovation of the spline), as shown on the bottom of Figure 1.
\begin{figure}[t]
\centerline{\includegraphics[width=0.5\linewidth]{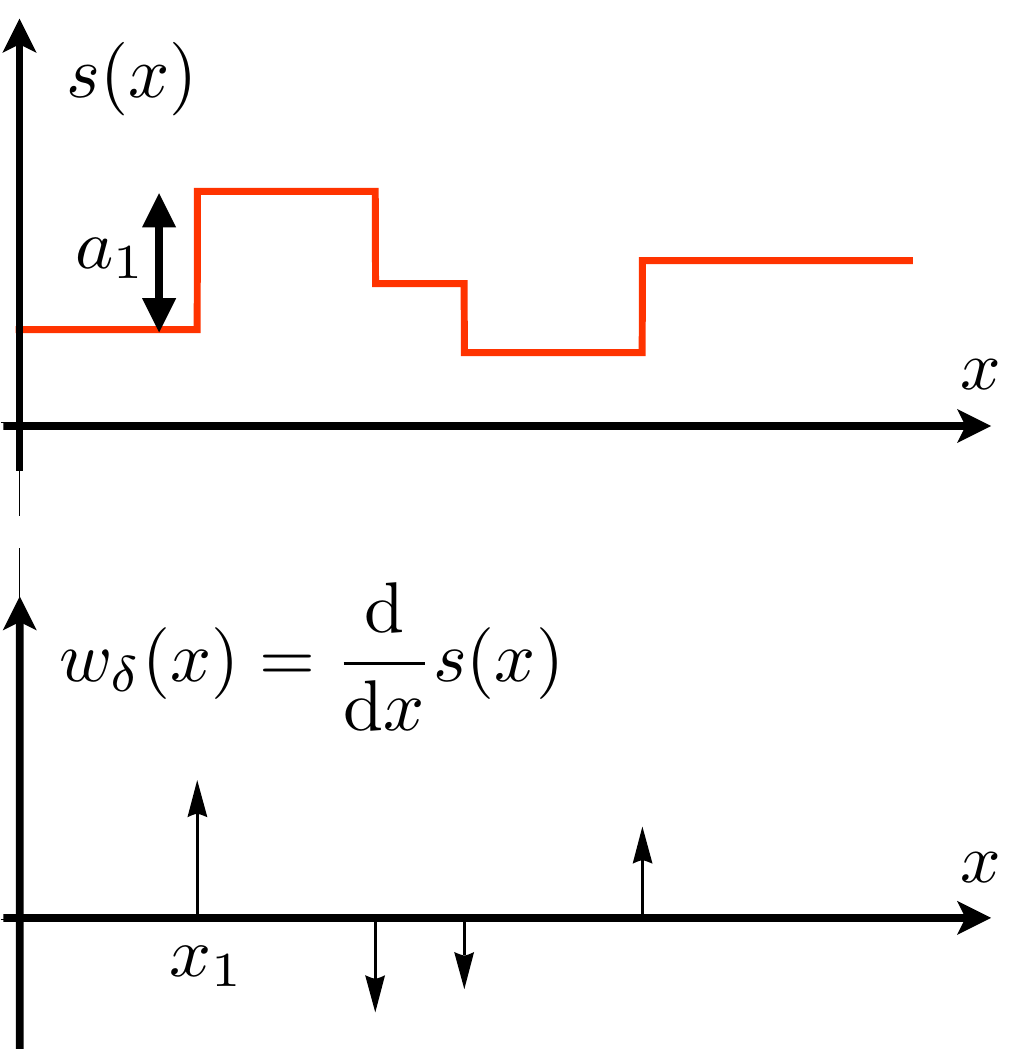}}
\caption{Prototypical solution of a linear inverse problem with total-variation regularization. The signal is piecewise-constant; in other words, it is a nonuniform $\Lop$-spline with $\Lop=\Dop$ (derivative operator). The application of $\Dop$ uncovers the innovation $w_\delta$: The Dirac impulses are located at the points of discontinuity (knots), while their height (weight) encodes the magnitude of the corresponding jump.\label{contantspline}}
\end{figure}
In view of Definition \ref{Def:spline}, the solution \eqref{Eq:Dspline} can also be described 
as a nonuniform $\Lop$-spline with $\Lop=\Dop$. The remarkable aspect of this result is that the parametric form \eqref{Eq:Dspline} is universal,
in the sense that it does not dependent on the measurement functionals $\nu_m$. To the best of our knowledge, this is the first mathematical explanation of the well-known observation that TV regularization tends to enforce piecewise-constant solutions. The other interesting point is that
one can interpret the solution as the best $K$-term representation of the signal within an infinite-dimensional
dictionary that consists of a constant signal $p_1$ plus a continuum of shifted Green's functions ({\em i.e.}, $\{\One_+(\cdot-\tau)\}_{\tau \in\R}$), making the connection with the synthesis views \eqref{Eq:CS1}  and \eqref{Eq:wav} of the problem.
Also, note that the described sparsifying effect is much more dramatic than that of the finite-dimensional setting since one is collapsing a continuum (integral representation) into a discrete and finite sum.
}

\review{We shall now show that the mechanism at play is very general and transposable to a much broader class of regularization operators $\Lop$ and data-fidelity terms, as well as for the multidimensional setting.
}

\review{
\subsection{Road Map of the Paper}
The remainder of this paper is organized as follows: After setting the \revise{notation}, we present and discuss of our main representer theorem (Theorem \ref{Theo:L1spline}) in Section \ref{Sec:RepTheo}. We also provide a refined version for the simpler interpolation scenario (Theorem \ref{Theo:L1spline2}). We then proceed with the review of 
primary applications in Section \ref{Sec:Appli}. 
}

\review{
The mathematical tools for proving our results are developed in the second half of the paper. The first enabling component is the tight connection between splines and operators, which is the topic of Section \ref{Sec:Splines}. In particular, we present an operator-based method to synthesize a spline from its innovation, which requires the construction of an appropriate right-inverse operator (Theorem \ref{Theo:inverse}).
The existence of such inverse operators is fundamental to the characterization of the native spaces associated
with our generalized total-variation criterion (gTV) (Theorem \ref{Theo:gBeppoLevi}), as we show in Section \ref{Sec:Space}. 
The actual proof of Theorems \ref{Theo:L1spline} and \ref{Theo:L1spline2} is given in Section \ref{Sec:ProofgTVTheorems}. It relies on a preparatory result (generalized Fisher-Jerome theorem) that establishes the impulsive form of the solutions of some abstract minimization problem over the space $\Spc M(\R^d)$  of bounded Borel measures. 
}

\review{
We conclude the paper in Section \ref{Sec:OpenProbs} with a brief discussion of open issues.
}
\section{Representer Theorems for Generalized Total Variation}
\label{Sec:RepTheo}
 \review{
Although we are considering a finite number of measurements, we are formulating the reconstruction problem in the continuous domain. This calls for a precise specification of the underlying functional setting.}
\subsection{\revise{Notation}} The space of tempered distribution is denoted by $\Spc S'(\R^d)$ where $d$ gives the number of dimensions. This space is made of continuous linear functionals $\mu: \varphi \mapsto \langle \mu, \varphi\rangle$ acting on the Schwartz' space $\Spc S(\R^d)$ of smooth
and rapidly decaying test functions on $\R^d$ \cite{Gelfand-Shilov1964, Hormander1990}.

\review{ We shall primarily work with
 the space $\Spc M(\R^d)$ of regular, real-valued, countably additive Borel measures on $\R^d$, which is also known (by the Riesz-Markov theorem) to be the continuous dual of $C_0(\R^d)$: the Banach space of continuous functions on $\R^d$ that vanish at infinity equipped with the supremum norm $\|\cdot\|_{\infty}$  \cite[Chap. 6]{Rudin1987}.
Since $\Spc S(\R^d)$ is dense in $C_0(\R^d)$, this allows us to define $\Spc M(\R^d)$ as
 \begin{align}
\Spc M(\R^d)=\{w \in \Spc S'(\R^d):\|w\|_{\Spc M}=\sup_{\varphi \in \Spc S(\R^d): \|\varphi\|_\infty=1} \langle w,\varphi\rangle<\infty\},
\label{Eq:DefM}
\end{align}
and also to extend the space of test functions to $\varphi \in C_0(\R^d)$. The action of $w$ will be denoted
by
$\varphi \mapsto \langle w,\varphi\rangle=\int_{\R^d} \varphi(\bx) w(\bx)\dint \V x$ where the right-hand side stands for the Lebesgue integral of $\varphi$ with respect to the underlying measure\footnote{The use of $w(\bx)\dint \V x$ in the integral is a slight abuse of notation when the measure is not absolutely continuous with respect to the Lebesgue measure. 
}. 
The bottom line is that $\Spc M(\R^d)$ is the Banach space associated with the norm 
$\|\cdot\|_{\Spc M}$ which returns the ``total variation'' of the measure that specifies $w$.}

Two key observations in relation to our goal are:
\begin{enumerate}
\item  the compatibility of the $L_1$ and total-variation norms with the former being stronger than the latter. Indeed, $\|f\|_{L_1(\R^d)}=\|f\|_{\Spc M}$ for all
$f \in L_1(\R^d)$;
\item the inclusion
of Dirac impulses in $\Spc M(\R^d)$, but not in $L_1(\R^d)$.
Specifically, $\delta(\cdot-\V x_0) \in \Spc M(\R^d)$ for any fixed offset $\V x_0 \in\R^d$ with $ \langle \delta(\cdot-\V x_0),\varphi\rangle=\varphi(\V x_0)$ for all $\varphi \in C_0(\R^d)$.
\end{enumerate}

We shall monitor the algebraic rate of growth/decay of (ordinary) functions of the variable $\V x \in \R^d$ by verifying their inclusion in the Banach space
\review{
\begin{align}
L_{\infty,\alpha_0}(\R^d)=\{f: \R^d\to \R \  \mbox{  s.t.  }  \ \|f\|_{\infty,\alpha_0}<+\infty\},
\label{Eq:Linfty}
\end{align}
where
$$\|f\|_{\infty,\alpha_0}=\esssup_{\bx \inR^d} \left(|f(\bx)|(1+\|\bx\|)^{-\alpha_0}\right)$$
with $\alpha_0 \in \R$. For instance, $\V x^\V m=x_1^{m_1}\cdots x_d^{m_d} \in L_{\infty,\alpha_0}(\R^d)$ for $\alpha_0\geq |\V m|=m_1+\dots+m_d$.
}

\review{
A linear operator whose output is a function is represented with a roman capital letter ({\em e.g.}, $\Lop$).
The action of $\Lop$ on the signal $s$ is denoted by $s \mapsto \Lop\{s\}$, or $\Lop s$ for short. Such an operator is said to be {\em shift-invariant}
if it commutes with the shift operator $s \mapsto s(\cdot-\bx_0)$; that is, if $\Op L\{s(\cdot-\bx_0)\}=\Op L\{s\}(\cdot-\bx_0)$ for any admissible signal $s$
and $\bx_0 \in\R^d$.
}
\subsection{Main Result on the Optimality of Splines}
 
 \review{Since the solution is regularized, the constrained minimization is performed over some native space $\Spc M_\Lop(\R^d)$
that is tied to some admissible differential operator $\Lop$, such as $\Dop$, $\Dop^2$ (second derivative), or $\Delta$ (Laplacian) for $d>1$.
}

 \begin{definition} [Spline-admissible operator]
\label{Def:splineadmis}
A linear operator $\Op L: \Spc M_\Lop(\R^d) \to \Spc M(\R^d)$, where $\Spc M_\Lop(\R^d)\supset \Spc S(\R^d)$ is an appropriate subspace of $\Spc S'(\R^d)$, is called {\em spline-admis\-sible}
if 
\begin{enumerate}
\item it is shift-invariant;
\item there exists a function $\rho_{\Op L}: \R^d \to\R $ of slow growth (the Green's function of $\Op L$) such that
$\Lop\{\rho_{\Op L}\}=\delta$, where $\delta$ is the Dirac impulse. The rate of polynomial growth of 
$\rho_{\Op L}$ is $n_0=\inf\{n \in \N: \rho_\Lop \in L_{\infty,n}(\R^d)\}$.
\item the (growth-restricted) null space of $\Lop$, $$\Spc N_\Lop=\{q \in L_{\infty,n_0}(\R^d): \Lop\{q\}=0\},$$
has the finite dimension $N_0\ge0$.
\end{enumerate}
\end{definition}


 \review{
 The native space of $\Lop$, $\Spc M_{\Lop}(\R^d)$, is then specified as
\begin{align}
\label{Eq:native}
\Spc M_{\Lop}(\R^d)=\{f \in L_{\infty, n_0}(\R^d): \|\Lop f\|_{\Spc M}<\infty\}.
\end{align}
It is largest function space for which the generalized total variation $$
{\rm gTV}(f)=\|\Lop f\|_{\Spc M}
$$
is well-defined under the finite-dimensional null-space constraint $$ \|\Lop f\|_{\Spc M}=0 \Leftrightarrow f \in \Spc N_{\Lop}, \mbox{  for any }f \in \Spc M_{\Lop}(\R^d).$$
This also means that ${\rm gTV}$ is only a semi-norm on $\Spc M_\Lop(\R^d)$. However, it can be turned into a proper norm by factoring out the null space of $\Lop$.
We rely on this property and the finite dimensionality of $\Spc N_\Lop$  to prove that $\Spc M_{\Lop}(\R^d)$ is a {\em bona fide} Banach space (see Theorem \ref{Theo:gBeppoLevi}).} 

{Having set the functional context, we now 
state our primary representer theorem.
}


\setcounter{theorem}{0}
\begin{theorem}[gTV optimality of splines for linear inverse problems]
\label{Theo:L1spline}
Let us as\-sume that the following conditions are met:
\begin{enumerate}
\item The regularization operator ${\Lop:\Spc M_{\Lop}(\R^d)\to\Spc M(\R^d)}$ is spline-admissible in the sense of Definition \ref{Def:splineadmis}. 
\item The linear measurement operator $\V \nu: f \mapsto \V \nu(f)=\big(\langle \nu_1, f\rangle, \ldots, \langle \nu_M, f\rangle\big)$ maps $\Spc M_{\Lop}(\R^d) \to \R^M$ and is weak*-continuous on $\Spc M_\Lop(\R^d)=\big( C_\Lop(\R^d)\big)'$.
\item \review{The recovery
problem is well-posed over the null space of $\Lop$: $\V \nu(q_1)=\V \nu(q_2) \Leftrightarrow q_1=q_2$, for any $q_1,q_2 \in \Spc N_\Lop$.}
\end{enumerate}
%
Then, the extremal points of the general constrained minimization problem
\begin{align}
\label{eq:genproblem}
\beta=\min_{f \in \Spc M_{\Lop}(\R^d)} \|\Lop f\|_{\Spc M}\quad \mbox{   s.t.   } \quad \V \nu(f) \in \Spc C,
\end{align}
where $\Spc C$ is any (feasible) convex compact subset of $\R^M$, are necessarily nonuniform $\Lop$-splines of the  form
\begin{align}
\label{eq:spline}
s(\bx)=\sum_{k=1}^K a_k \rho_\Lop(\bx-\bx_k)+\sum_{n=1}^{N_0} b_n p_n(\bx)
\end{align}
with parameters $K\le M$, $\{\bx_k\}_{k=1}^K$ with $\V x_k \in \R^d$,  $\M a=(a_1,\ldots,a_K)\in \R^K$, and $\M b=(b_1,\ldots,b_{N_0})\in \R^{N_0}$. Here,
$\{p_n\}_{n=1}^{N_0}$ is a basis of $\Spc N_\Lop$ and $\Lop\{\rho_\Lop\}=\delta$ so that $\beta=\|\Lop s\|_{\Spc M}=\sum_{k=1}^K |a_k|=\|\M a\|_1$. The full solution set of \eqref{eq:genproblem} is the convex hull of those extremal points.
\end{theorem}

Theorem \ref{Theo:L1spline} is a powerful existence result that points towards the universality of 
nonuniform $\Lop$-spline solutions. \review{ The key property here is $\Lop\{s\}=\sum_{k=1}^K a_k \delta(\cdot-\bx_k)$, which follows from Conditions 1-3 in Definition \ref{Def:splineadmis} and is consistent with 
the more detailed characterization of splines presented in Section \ref{Sec:Splines}.
For the time being, it suffices to remark that these splines are smooth ({\em i.e.}, infinitely differentiable) everywhere, except at their knot locations $\{\bx_k\}$.}

Although the extremal problem is defined over a continuum, the remarkable outcome is that the
problem admits solutions that are intrinsically sparse, with the level of sparsity being measured by the minimum number $K$ of required spline knots.
In particular, this explains why the solution of a problem with a TV/$L_1$-type constraint on the derivative (resp., the second derivative) is piecewise-constant (resp., piecewise linear when $\Lop=\Dop^2$) with breakpoints at $x_k$. The other pleasing aspect is the direct connection between the functional concept of generalized TV and the $\ell_1$-norm of the expansion coefficients $\M a$.

We observe that the solution is made up of two components: an adaptive one that is specified by $\{\bx_k\}$ and $\M a$, and a linear regression term (with expansion coefficients $\M b$) that describes the component in the null space of the operator. Since $\M b$ does not contribute to $\|\Lop s\|_{\Spc M}$, the optimization tends to maximize the
contribution of the null-space component. The main difficulty in finding the optimal solution is that $K$ and $(\bx_k)$ are problem-dependent and unknown {\em a priori}. 

\review{
We have mentioned in Section \ref{Sec:linkTV} that the semi-norm $\|\Dop f\|_{\Spc M}$ yields the classical total variation of a function in 1D. Unfortunately, there is no such direct connection for $d>1$, the reason being that the multidimensional gradient $\V \nabla$ is not spline-admissible because it is a vector operator. Instead, as a proxy for the popular total variation of Rudin and Osher \cite{Rudin1992}, we suggest using the (fractional) Laplacian semi-norm $\|(-\Delta)^{\gamma/2}f\|_{\Spc M}$ with $\gamma\ge d$, which is endowed with the same invariance and null-space properties. According to Theorem \ref{Theo:L1spline}, such a $\gamma$th-order regularization results in extremal points that are nonuniform polyharmonic splines \cite{Duchon1977,Madych1990}.}

\subsection{Connection with Unconstrained Problem}
The statement in Theorem \ref{Theo:L1spline} is remarkably general. In particular, it covers the
generic regularized least-squares problem
\begin{align}
\label{eq:l1min}
f_\lambda=\arg \min_{f \in \Spc M_{\Lop}(\R^d)} \left(\sum_{m=1}^M |y_m - \langle \nu_m,f\rangle|^2 + \lambda \|\Lop f\|_{\Spc M} \right),
\end{align}
which is commonly used to formulate linear inverse/compressed-sensing problems \cite{Bruckstein2009, Candes2007, Donoho2006, Elad2010b, Figueiredo2003}. 
The connection is obtained by taking $\Spc C=\{\M z \in \R^M: \|\M y -Ê\M z\|_2^2\le \epsilon^2\}=B(\M y; \epsilon)$,
which is a ball of diameter $\epsilon$ centered on the measurement vector $\M y=(y_1,\ldots,y_M)$.
Indeed, since the data-fidelity term is (strictly) convex, the extreme points $s_\epsilon$ of \eqref{eq:genproblem} saturate the inequality such that $\|\M y-\V \nu(s_\epsilon)\|_2^2=\epsilon^2$ and gTV is minimized with $\alpha=\alpha(\epsilon)=\|\Lop s_\epsilon\|_{\Spc M}$. In the unconstrained form \eqref{eq:l1min}, the selection of a fixed $\lambda  \in \R^+$ results in a particular value of the data error $\|\M y-\V \nu(f_\lambda)\|_2^2=\epsilon'(\lambda)$ with the optimal solution $f_\lambda=s_{\epsilon'}$ having the same total variation as if we were looking at the primary problem \eqref{eq:genproblem}  with $\Spc C=B(\M y; \epsilon')$. 

To get further insights on the optimization problem \eqref{eq:l1min}, we can look at two limit cases.
When $\lambda\to \infty$, the solution must take the form $f_\infty=p \in \Spc N_\Lop$ so that $ \|\Lop f_\infty\|_{\Spc M}=0$.
It then follows that  $\|\M y-\V \nu(f_\infty)\|_2^2\le\|\M y\|^2<\infty$. 
On the contrary, when $\lambda\to 0$, the minimization will force the data term $\|\M y-\V \nu(f_0)\|_2^2$ to vanish. Theorem \ref{Theo:L1spline} then ensures the existence of a nonuniform  ``interpolating''
$\Lop$-spline $f_0(\bx)$ with $\V \nu(f_0)=\M y$ and minimum gTV semi-norm. 

\review{\subsection{Generalized Interpolation}
In the latter interpolation scenario, the convex set $\Spc C$ reduces to the single point $\Spc C=\{\M y \in \R^M\}$. This configuration is of special theoretical relevance because it enables us to refine our upper bound on the number $K$ of spline knots.
\setcounter{theorem}{1}
 \begin{theorem}[Generalized spline interpolant]
\label{Theo:L1spline2}
Under Assumptions 1-3 of Theorem \ref{Theo:L1spline}, the extremal points of the (feasible) generalized interpolation problem
\begin{align}
\label{eq:genproblem2}
\arg \min_{f \in \Spc M_{\Lop}(\R^d)} \|\Lop f\|_{\Spc M}\quad \mbox{   s.t.   } \quad \V \nu(f)=\M y
\end{align}
are nonuniform $\Lop$-splines of the same form \eqref{eq:spline} as in Theorem \ref{Theo:L1spline}, but with $K\le (M-N_0)$.
\end{theorem}
}
\section{Application Areas}
\label{Sec:Appli}

\review{ 
We first briefly comment on the
admissibility conditions in Theorem \ref{Theo:L1spline} and indicate that the restrictions are minimal. 
To the best of our knowledge, the continuity of the measurement operator $\V \nu$ is a necessary requirement for the mathematical analysis of any inverse problem. The difficulty here is that our native space $\Spc M_\Lop(\R^d)=(C_\Lop(\R^d))'$ is non-reflexive, which forces us to rely on the 
weak*-topology. The continuity requirement in Theorem 1 is therefore equivalent to $\nu_m \in C_\Lop(\R^d)$ for $ m=1,\dots,M$ where the Banach structure of the predual space $C_\Lop(\R^d)$ is laid out in Theorem \ref{Theo:Predual}. In particular, we refer to the norm inequality
\eqref{EQ:NormInecL1}, which suggests that Condition 2 in Theorem \ref{Theo:L1spline} is met by picking $\nu_m \in L_{1,-n_0}(\R^d)$ where the latter is the Banach space associated with the weighted $L_1$-norm
\begin{align}
\label{Eq:L1n0}
\|f\|_{L_{1,-n_0}}=\int_{\R^d} |f(\bx)|(1+\|\bx\|)^{n_0}\dint \bx.
\end{align}
\revise{In fact, $L_{1,-n_0}(\R^d)$ is the predual of the space $L_{\infty,n_0}(\R^d)$ defined by \eqref{Eq:Linfty}, which implies that $\Spc M_\Lop(\R^d)=\big(C_\Lop(\R^d)\big)'\subset \big(L_{1,-n_0}(\R^d) \big)'=L_{\infty,n_0}(\R^d)$.} 
The condition $\nu_m \in L_{1,-n_0}(\R^d)$ is a mild algebraic decay requirement that turns out to be satisfied by the impulse response of most physical devices.
As for the requirement that the inverse problem is well defined over the null space of $\Lop$ (Condition 3), it a prerequisite to the success of any regularization scheme. Otherwise, there is simply no hope of turning an ill-posed problem into a well-posed one.
For instance, in the introductory example with classical total-variation regularization, the constraint is that $\V \nu$ should have at least one component $\nu_m$ such that 
$\langle \nu_m,1\rangle\ne 0$ which, again, is very mild requirement.}

Next, we discuss examples of signal recovery that are covered by Theorems \ref{Theo:L1spline} and \ref{Theo:L1spline2}.
The standard setting is that one is given a set of noisy measurements $\M y=\V \nu(s) + {\rm ``noise"}$ of an unknown signal $s$ and that one is trying to recover $s$ from $\M y$ based on the solution of \eqref{eq:l1min}, or some variant of the problem involving some other (convex) data term---the most favorable choice being the log likelihood of the measurement noise. We shall then close the discussion section by briefly making the connection with a class of inverse problems in measure space; that is, the case $\Lop={\rm Identity}$.
%
%
%
\subsection{Interpolation}
The task here is to reconstruct a continuous-domain signal from its (possibly noisy) nonuniform samples $\{s(\bx_m)\}_{m=1}^M$, which is achieved by searching for the function $s: \R^d \to \R$ that fits the samples while minimizing $\|\Lop s\|_{\Spc M}$.
This corresponds to the problem setting in Theorem \ref{Theo:L1spline} with $\nu_{m}=\delta(\cdot-\bx_m)$ and $\Spc C=B(\M y; \epsilon)$, where $\M y$ denotes the measurement vector. Hence, the admissibility condition $\nu_m \in C_\Lop(\R^d)$ 
is equivalent to $\revise{(\Lop_{\V \phi}^{-1})^\ast\{\delta(\cdot-\bx_m)\}}=g_{\V \phi}(\V x_m, \cdot) \in C_{0}(\R^d)$, where the boundedness is ensured by the stability condition in Theorem \ref{Theo:inverse}. The more technical continuity requirement is achieved when $\rho_\Lop$ is continuous (H\"older exponent $r_0>0$). This happens when the order of the differential operator is greater than one, which seems to exclude\footnote{We can bypass this somewhat artificial limitation by replacing the ideal sampler by a quasi-ideal sampling device that involves a mollified version of a Dirac impulse.}  
simple operators such as $\Dop$ (piecewise-constant approximation). This limitation notwithstanding, our theoretical results are directly applicable to the problems of adaptive regression splines \cite{Mammen1997} with $\Lop=\Dop^N$, 
the construction of shape-preserving splines \cite{Lavery2000}, as well as a whole range of variations including TV denoising.  \\[-4ex]
%
\subsection{Generalized Sampling} 
The setting is analogous to the previous one, except that the samples are now observed through a sampling aperture 
$\phi\in L_{1,-n_0}(\R^d)$ so that $\nu_{m}={\phi(\cdot-\bx_m)}$ \cite{Elda2015,Unser2000}. The function $\phi$ may, for example, correspond to the point-spread function of a microscope. Then, the recovery problem is equivalent to a deconvolution \cite{Dey2006}. Since the measurements are obtained by integration of $s$ against an ordinary function $\nu_m\in L_{1,-n_0}(\R^d)$, there is no
requirement for the continuity of $\rho_\Lop$ because of the implicit smoothing effect of $\phi$.
This means that essentially no restrictions apply. \\[-4ex]
\subsection{Compressed Sensing}  
The result of Theorem \ref{Theo:L1spline} is highly relevant to compressed sensing, especially since the underlying $L_1$/TV signal-recovery problem is formulated in the continuous domain. We like to view \eqref{eq:spline} as the prototypical form of a piecewise-smooth signal
that is intrinsically sparse with sparsity $K=\|\V a\|_0$. The model also conforms with the notion of a finite rate of innovation  \cite{Vetterli2002}.
If we know that the unknown signal $s$ has such a form, then Theorem  \ref{Theo:L1spline} suggests that we can attempt to recover it from an $M$-dimensional linear measurement $\M y= \V \nu(s)$ by solving the optimization problem \eqref{eq:genproblem} with $\Spc C=B(\V y;\epsilon)$, which is in agreement with the predominant paradigm in the field. 
While the theorem states that $M\ge K$, common sense dictates that we should take $M > N_{\rm freedom}$, where $N_{\rm freedom}=2 K + N_0$ is the number of degrees of freedom of the underlying model. The difficulty, of course, is that a subset of those parameters (the spline knots $\V x_k$) induce a model dependency that is highly nonlinear. 

\subsection{Inverse Problems in the Space of Measures}
\revise{Some of the theoretical results of this paper are also of direct relevance for inverse problems that are formulated in the space $\Spc M(\R^d)$ of measures \cite{Bredies2013}.
The prototypical example is the recovery of the location (with super-resolution precision) of a series of Dirac impulses from noisy measurements, which may be achieved through the continuous-domain minimization of the total variation of the underlying measure \cite{Candes2013b,Denoyelle2015support,Duval2015,Fernandez2016}.
The Fisher-Jerome theorem \cite[Theorem 1]{Fisher1975} as well as our extension for the unbounded domain $\R^d$ and arbitrary convex sets (Theorem \ref{Theo:gFisher}) support this kind of algorithm, as they guarantee the existence of sparse solutions---understood as a sum of Dirac spikes---for this family of problems.}

\section{Splines and Operators}
\label{Sec:Splines}
\review{We now switch to the explanatory part of the presentation. The first important concept that is implicit in the statement of
Theorems \ref{Theo:L1spline} and \ref{Theo:L1spline2} is the powerful association between splines and operators, the idea being that the selection of an admissible operator $\Lop$  specifies a corresponding type of splines  \cite{Schultz1967,Schumaker2007}\cite[Chapter 6]{Unser2014book}.}

Sorted by increasing complexity, the three types of operators that are of relevance to us are: 
(i) ordinary differential operators, which are polynomials of the derivative operator $\Dop=\frac{\dint}{\dint x}$ \cite{Dahmen:Micchelli:1987, Schultz1967, Unser2005}; (ii) partial differential operators such as the Laplacian $\Delta$ (or some polynomial thereof); and (iii) fractional derivatives such as $\Dop^\gamma$ or $(-\Delta)^{\frac{\gamma}{2}}$ with $\gamma\in \R^+$ whose Fourier symbols are $(\jj \omega)^\gamma$ and $\|\bw\|^\gamma$, respectively \cite{Duchon1977,Unser2000a, Unser2007}.
It can be shown that all linear-shift-invariant operators of Type (i) and all elliptic operators of Type (ii) are spline-admissible in the sense of Definition \ref{Def:splineadmis}. \review{This is also known to be true for fractional derivatives and fractional Laplacians with $\gamma\ge d$ \cite{Duchon1977,Unser2000a}.}

\review{Let us mention that the issue of making sure that the null space of the operator $\Lop$ is finite-dimensional is often nontrivial for $d>1$. It is a fundamental aspect that is addressed in the $L_2$ theory of radial basis functions and polyharmonic splines with the definition of the appropriate native spaces \cite[Chapter 10]{Wendland2005}.
Here, we have chosen to bypass some of these technicalities by including a growth restriction \big({\em i.e.}, the condition that $q\in L_{\infty,n_0}(\R^d)$\big) in the definition of $\Spc N_\Lop$}. A fundamental property in that respect is that 
the finite-dimensional null space of a LSI operator can only include exponential polynomial components of the form
$\bx^{\V m}\ee^{\jj \langle\bw_0,\bx\rangle}$, which correspond to a zero of multiplicity at least $|\V m|+1$ of the frequency response $\widehat L(\bw)$ at $\bw=\bw_0$ (see \cite[Proposition 6.1 p. 118]{Unser2014book} \review{and \cite[Section 6]{HelOr1998}}).


Once it is established that $\Lop$ is spline-admissible, one can rely on the following unifying distributional definition of a spline.

\setcounter{theorem}{1}
\begin{definition} [Nonuniform $\Lop$-spline] 
\label{Def:spline}
A function $s: \R^d \to \R$ of slow growth (i.e., $s \in L_{\infty,n_0}(\R^d)$ with $n_0\ge0$) is said to be a nonuniform $\Lop$-spline if
\begin{align}
\label{Eq:inov}
\Lop\{s\}=\sum_{k} a_k\delta(\cdot-\bx_k)=w_\delta,
\end{align}
where $(a_k)$ is a sequence of weights and the Dirac impulses are located at the spline knots 
$\{\bx_k\}$.  
\end{definition}
\review{The generalized function $\Lop\{s\}=w_\delta$ is called the innovation of the spline because it contains the crucial information for its description: the positions $\{\bx_k\}$ of the knots and the amplitudes $(a_k)$ of the corresponding discontinuities.}

\review{The one-dimensional brands of greatest practical interest are the polynomial splines with $\Lop=\Dop^m$ \cite{deBoor1978, Schoenberg1946} and the exponential splines \cite{Dahmen:Micchelli:1987, Schultz1967, Unser2005} with $\Lop=c_m \Dop^m + \dots + c_1 \Dop + c_0 \Op I$, where $\Op I=\Dop^0$ denotes the identity operator. Their multidimensional counterparts are the polyharmonic splines with
$\Lop=(-\Delta)^{\gamma/2}$ \cite{Duchon1977,Madych1990} and the Sobolev splines with $\Lop=(\Op I -\Delta)^{\gamma/2}$ 
for $\gamma\ge d$ \cite{Ward20l4}. The connection with the theory of Sobolev spaces is that the Green's functions of $(-\Delta)^{\gamma/2}$ (resp., $(\Op I -\Delta)^{\gamma/2})$ are the kernels of the Riesz (resp., Bessel) potentials \cite{Grafakos2004}. }

\review
{For a constructive use of Definition \ref{Def:spline}, we also need to be able to re-synthesize the spline $s$ from its innovation.
In the case of our introductory example with $\Lop=\Dop$ (see Figure \ref{contantspline}), one simply integrates $w_\delta$, which yields \eqref{Eq:Dspline} (up to the integration constant $b_1$) owing to the property that
$\Dop^{-1}\{\delta(\cdot-x_k)\}(x)=\int_{-\infty}^x \delta(\tau-x_k)\dint \tau =\One_+(x-x_k)$.
In principle, the same inversion procedure is applicable for the generic operator $\Lop$ and amounts to substituting the $\delta$ distribution in \eqref{Eq:inov}
by the Green's function $\rho_\Lop$.}
The only delicate part is the proper handling of the ``integration constants'' (the part of the solution that lies in the null space of the operator), which is achieved through the specification of $N_0$ linear boundary conditions of the form $\langle \phi_n,s\rangle=0$.
%

\review{We now show that the underlying functionals $\V \phi=(\phi_1,\ldots,\phi_{N_0})$ can be incorporated in the specification of an appropriate right-inverse operator $\Lop^{-1}_\V \phi$.}
\review{Our construction requires that $\V \phi$ first be matched to a basis of $\Spc N_\Lop$ such as to form a biorthogonal system. We note that this is always feasible as long as the $\phi_n$ are linearly independent with respect to $\Spc N_\Lop$. (An explicit construction is given in the proof of Theorem \ref{Theo:L1spline2}.)
\setcounter{theorem}{2}
\begin{definition}
\label{Def:basis}
\revise{The pair} $(\V \phi, \V p)=(\phi_n,\ p_n)_{n=1}^{N_0}$ is called a biorthogonal system for $\Spc N_\Lop\subset\Spc M_\Lop(\R^d)$ if
$\ \{p_n\}_{n=1}^{N_0}$ is a basis of $\Spc N_\Lop$ and the vector of ``boundary'' functionals $\V \phi=(\phi_1,\ldots,\phi_{N_0})$ with $\phi_n \in  \Spc N'_\Lop$
satisfy the biorthogonality condition
$\V \phi(p_n)=\M e_n$ 
where $\M e_n$ is the $n$th element of the canonical basis.
\end{definition}}

\review{
The interest of such a system is that any $q \in \Spc N_{\Lop}$ has a unique representation as
$q=\sum_{n=1}^{N_0} \langle \phi_n, q\rangle p_n$ with associated norm $\|\V \phi(q)\|_2$.
}

\review{The fundamental requirement for our formulation is the stability/continuity of the inverse operator $\Lop^{-1}_\V \phi : \Spc M(\R^d) \to \Spc M_\Lop(\R^d)$. Since $\Spc M_\Lop(\R^d) \subset L_{\infty,n_0}(\R^d)$ by construction, we can control stability by relying on  Theorem \ref{Theo:stableinverse} whose proof is given in Appendix \ref{App:Stable}. 
\setcounter{theorem}{2}
\begin{theorem} 
\label{Theo:stableinverse}
\review{The generic linear operator $\Op G: w \mapsto f=$ ${\int_{\R^d} g(\cdot,\V y) w(\V y)\dint \V y}$
continuously maps $\Spc M(\R^d) \to L_{\infty,\alpha_0}(\R^d)$ with $\alpha_0 \in \R$ if and only if its kernel $g$ is measurable 
and
\begin{align}
\label{Eq:Stable}
\esssup_{\bx, \V y \in \R^d} \left(|g(\V x,\V y)|\;(1+\|\bx\|)^{-\alpha_0}\right) < \infty.
\end{align}}
\end{theorem} }

\review{This allows us to characterize the desired operator in term of its Schwartz' kernel  (or generalized impulse response) $g_{\V \phi}(\V x, \V y)=\Lop^{-1}_\V \phi\{\delta(\cdot-\V y) \}(\V x)$.}

\setcounter{theorem}{3}
\review{\begin{theorem}[Stable right-inverse of $\Lop$]
\label{Theo:inverse}
Let $(\phi_n,\ p_n)_{n=1}^{N_0}$ be a biorthogonal system  for $\Spc N_\Lop \subset \Spc M_\Lop(\R^d) \subset L_{\infty,n_0}(\R^d)$. Then, there exists a unique operator $\Lop^{-1}_\V \phi: \varphi \mapsto \Lop^{-1}_\V \phi \varphi=\int_{\R^d} g_{\V \phi}(\cdot ,\V y) \varphi(\V y)\dint\V y$ such that
\begin{align}
\label{Eq:rightinv}\Lop\Lop_\V \phi^{-1}\varphi = \varphi & \qquad(\mbox{right-inverse property})\\
\label{Eq:boundary}\V \phi(\Lop_\V \phi^{-1}\varphi) =\V 0& \qquad(\mbox{boundary condidions})
\end{align}
for all $\varphi \in \Spc S(\R^d) $.
The kernel of this operator is 
\begin{align}
\label{Eq:kernelinverseg}
g_{\V \phi}(\V x,\V y) 
&=\rho_{\Lop}(\V x-\V y)- \sum_{n=1}^{N_0}  p_n(\bx) q_n(\V y),
\end{align}
with $\rho_\Lop$ such that $\Lop\rho_{\Lop}=\delta$ and $q_n(\V y)=\langle \phi_n,\rho_{\Lop}(\cdot-\V y) \rangle$. Moreover, if $g_{\V \phi}$ satisfies the stability condition \eqref{Eq:Stable} with $\alpha_0=n_0$, then $\Lop^{-1}_{\V \phi}$ admits a continuous extension $\Spc M(\R^d)\to L_{\infty,n_0}(\R^d)$ with \eqref{Eq:rightinv} and \eqref{Eq:boundary} remaining valid for all $\varphi \in \Spc M(\R^d)$.
%
%
\end{theorem}
The proof of Theorem \ref{Theo:inverse} 
is given in Appendix  \ref{App:Inverse}.}

\review{
Since the choice of the $N_0$ linear boundary functionals $\phi_n$ is essentially arbitrary, there is flexibility in defining admissible inverse operators.
The important ingredient for our formulation is the existence of such inverses with the unconditional guarantee of their stability (see Theorem \ref{Theo:gBeppoLevi} below). 
}

\review{
To put this result into context, we now provide some illustrative examples.
For $\Lop=\Dop^{N_0}$, we have that $n_0=(N_0-1)$, $\rho_{\Dop^{N_0}}(x)=\frac{x_+^{n_0}}{n_0!}$, and $p_n(x)=\frac{x^{n-1}}{(n-1)!}$ for $n=1,\ldots,N_0$, where the polynomial basis is biorthogononal to $\V \phi$ with $\phi_n(x)=(-1)^{(n-1)}\delta^{(n-1)}(x)$. This (canonical) choice of boundary functionals then translates into the construction of an inverse operator
$\Lop^{-1}_{\V \phi}$
that imposes the vanishing of the function and its derivatives at the origin. By applying \eqref{Eq:kernelinverseg} and recognizing the binomial expansion of $(x-y)^{n_0}$, we simplify the expression of the kernel of this operator as
\revise{\begin{align*}
g_{\V \phi}(x,y)&=\frac{(x-y)_+^{n_0}}{n_0!}-\sum_{n=0}^{n_0}\frac{x^{n}}{n!}\frac{(-y)_+^{n_0-n}}{(n_0-n)!}\\
&=
\left\{
\begin{array}{rc}
\frac{(x-y)^{n_0}}{n_0!} \One_{(0,x]}(y),  &  x\ge 0 \\
 -\frac{(x-y)^{n_0}}{n_0!} \One_{(x,0]}(y),   &   x<0.
\end{array}
\right.
\end{align*}}
The crucial observation here is that the function $g_{\V \phi}(x,\cdot)$ with $x\in \R$ fixed is compactly supported and bounded. Moreover, $\|g_{\V \phi}(x,\cdot)\|_{\infty}=|g_{\V \phi}(x,0)|=\frac{x^{n_0}}{n_0!}$ so that $g_\V \phi$ obviously satisfies the stability bound \eqref{Eq:Stable} with $\alpha_0=n_0$. By contrast, the condition fails for the conventional shift-invariant inverse $\varphi \mapsto \rho_{\Dop^{N_0}} \ast \varphi$ ($n_0$-fold integrator), which stresses out the non-trivial
stabilizing effect of the second correction term in \eqref{Eq:kernelinverseg}. The other important consequence of the correction is the vanishing of $g_{\V \phi}(x,y)$ as $y \to \pm \infty$ for any fixed $x \in \R^d$, contrary to its leading term $(x-y)_+^{n_0}/n_0!$ which does not decay (and even grows) as $y\to-\infty$. 
}

\review{
The primary usage of the inverse operators of Theorem \ref{Theo:inverse} is the resolution of differential equations of the form
\begin{align}
\label{Eq:Ls}
\Lop s=w \quad \mbox{ s.t. } \quad \V \phi(s)=(b_1,\dots,b_{N_0})
\end{align}
for some $w \in \Spc M(\R^d)$.
Indeed, by invoking the properties of $\Op L_{\V \phi}^{-1}$ and the biorthogonality of $(\V \phi ,\V p)$, we readily show that \eqref{Eq:Ls} admits a unique solution in $\Spc M_\Lop(\R^d)$, which is given by
$$
s=\Lop_{\V \phi}^{-1}w + \sum_{n=1}^{N_0} b_n p_n.
$$
For the particular case of the spline innovation $w_\delta$ in Definition \ref{Def:spline}, we find that
$$
s=\sum_{k}a_k\Lop_{\V \phi}^{-1}\{\delta(\cdot-\bx_k)\}+ \sum_{n=1}^{N_0} b_n p_n
$$
which, upon substitution of the kernel given by \eqref{Eq:kernelinverseg}, results in a form that is the same as \eqref{eq:spline} in Theorem \ref{Theo:L1spline} modulo some adjustment of the constants $b_n$.}
\section{Native or Generalized Beppo-Levi Spaces}
\label{Sec:Space}
The search for the solution of our optimization problem is performed over the native space $\Spc M_\Lop(\R^d)$ defined by 
\eqref{Eq:native}, which is the largest space over which our gTG regularization functional is well defined.
\review{The delicate aspect is that $\Spc M_\Lop(\R^d)$
is specified in terms of a semi-norm, in analogy with  the definition of the classical Beppo-Levi spaces of order $n\in \N$ and exponent $p\ge1$, written as
$\Spc B_{p,n} (\R^d)=\{f\in \Spc S'(\R^d): \| \partial^{\V m} f \|_{L_p}<\infty \mbox{ for all multi-indices } |\V m|=n\}$ \cite{Deny1954,Kurokawa1988}. Hence, in 1D, the proposed definition of $\Spc M_{\Dop^n}(\R)$ is a slight extension of  $\Spc B_{1,n} (\R)$. In higher dimensions, it can be shown\footnote{The argument is that the only functions that are harmonic and of slow growth are polynomials.} that $\Spc B_{p,2n} (\R^d)=\{f\in L_{\infty,2n-1}(\R^d): \|(-\Delta)^n f\|_{L_p}<\infty\}$, 
so that there also exists a close connection between $\Spc B_{1,2n} (\R^d)$ and  $\Spc M_{(-\Delta)^n} (\R^d)$. }
 
The crucial point for our formulation is that 
 $\Spc M_\Lop(\R^d)$ also happens to be a complete normed (or Banach) space when equipped with the proper direct-sum topology. We shall now make this structure explicit with the help of the inverse operators defined in Theorem \ref{Theo:inverse}. 
Since the principle is similar
to the characterization of the Beppo-Levi spaces, we shall also refer to $\Spc M_\Lop(\R^d)$ as a generalized Beppo-Levi space.  \\[1ex] 
%

\setcounter{theorem}{4}
\begin{theorem}[Banach-space structure of native space] 
\label{Theo:gBeppoLevi}
\review{Let $\Lop$ be a spline-admis\-sible operator, $\Spc M_\Lop(\R^d)$ its native space defined by \eqref{Eq:native}, and $(\V \phi,\V p)$ some biorthogonal system for its null space
$\Spc N_\Lop$. 
Then, the following equivalent conditions hold:}
\begin{enumerate}
\item \review{The right-inverse operator
$\Lop_\V \phi^{-1}$ specified by Theorem \ref{Theo:inverse} isometrically maps $\Spc M(\R^d) \to \Spc M_\Lop(\R^d)\subset L_{\infty,n_0}(\R^d)$, while its kernel necessarily fullfills the stability condition 
\begin{align}
\label{Eq:Stabilitybound}
C_{\V \phi}=\sup_{\bx, \V y \in \R^d} \left(|g_{\V \phi}(\V x,\V y)|\;(1+\|\bx\|)^{-n_0}\right)<\infty.\end{align}}
\item Every $f \in \Spc M_\Lop(\R^d)$ admits a unique representation as
$$
f=\Lop_\V \phi^{-1}w +p,
$$
where $w=\Lop\{f\} \in \Spc M(\R^d)$ and $p=\sum_{n=1}^{N_0} \langle f, \phi_n \rangle p_n  \in \Spc N_{\Lop}$.
\item $\Spc M_\Lop(\R^d)$ is a Banach space equipped with the norm
\begin{align}
\label{Eq:norm}
\|f\|_{\Lop, \V \phi}=\|\Lop f\|_{\Spc M}+\|\V \phi(f)\|_2.
\end{align}
\end{enumerate}
\end{theorem}

\begin{proof}
\review{
As preparation, we define a subset of $\Spc M_{\Lop}(\R^d)$ as
\label{App:gBeppo}
\begin{align}
\label{Eq:Mspace}
\Spc M_{\Lop,\V \phi}(\R^d)&=\{f \in \Spc M_\Lop(\R^d): \V \phi(f)=\M 0\}.
\end{align}
Since the boundary conditions $\V \phi(f)=\M 0$ are linear, $\Spc M_{\Lop,\V \phi}(\R^d)$ is clearly a vector space.
We now show that it is a Banach space when equipped with the norm $\|\cdot\|_\Lop=\|\Lop\{\cdot\}\|_{\Spc M}$. By definition, $\|\cdot\|_\Lop$ is a semi-norm on $\Spc M_\Lop(\R^d)$, meaning that it fulfills the properties of a norm, except for the unicity condition. To establish the latter on $\Spc M_{\Lop,\V \phi}(\R^d)$, we consider
$f \in \Spc M_{\Lop,\V \phi}(\R^d)$ such that $\|f\|_{\Lop}=0$, which is equivalent to $f \in \Spc N_\Lop$. Since $f\in \Spc M_{\Lop,\V \phi}(\R^d)$ (by hypothesis), we have that $\V \phi(f)=\M 0$, from which we deduce that $f=\sum_{n=1}^{N_0} \langle \phi_n,f\rangle p_n=0$, as expected. This proves that $\Spc M_{\Lop,\V \phi}(\R^d)$ is isometrically isomorphic to $\Spc M(\R^d)$ and, hence, a Banach space. 
Alternatively, one can also view $\Spc M_{\Lop,\V \phi}(\R^d)$ as a concrete transcription (or representative within the equivalence class) of the abstract quotient space $\Spc M_{\Lop}(\R^d)/\Spc N_\Lop$.
}

\vspace*{1ex}
\review{
{\em 1. Existence and Stability of Inverse Operators.} We have just revealed that $\Lop$ is a bijective, norm-preserving mapping $\Spc M_{\Lop,\V \phi}(\R^d) \to \Spc M(\R^d)$. This allows us to invoke the bounded-inverse theorem, which ensures the existence and boundedness (here, an isometry) of the inverse operator $\Lop^{-1}: \Spc M(\R^d) \to \Spc M_{\Lop,\V \phi}(\R^d)$. The relevant $\Lop^{-1}$ is precisely the unique operator $\Lop_{\V \phi}^{-1}$ identified in Theorem \ref{Theo:inverse}, as it imposes the boundary condition $\V \phi(\Lop_{\V \phi}^{-1}w)=\M 0$ for all $w\in \Spc M(\R^d)$. Finally, we use the same technique as in the proof of Theorem \ref{Theo:stableinverse} to establish the necessity of the stability condition \eqref{Eq:Stable} with $\alpha_0=n_0$.
%
}

\vspace*{0.5ex}{\em 2. Direct Sum Decomposition}. Since the system $(\V \phi, \V p)$ is biorthogonal,
the operator ${\rm Proj}_{\Spc N_\Lop}: f \mapsto \sum_{n=1}^{N_0}
\langle \phi_n, f \rangle p_n$ is a continuous projection operator $\Spc M_{\Lop}(\R^d) \to \Spc N_\Lop(\R^d)$. It follows that any element $f\in \Spc M_{\Lop}(\R^d)$
has a unique decomposition as
$
f=f_1 + q
$,
where $q={\rm Proj}_{\Spc N_\Lop}\{f\} \in \Spc N_\Lop$ and $f_1=(f-q)$ with $\V \phi(f_1)=\M 0$. This last condition implies that 
$f_1 \in \Spc M_{\Lop,\V \phi}(\R^d)$ so that $f_1$ has a unique representation as
$f_1=\Lop_{\V \phi}^{-1}w$, where $w=\Lop f_1=\Lop f \in \Spc M(\R^d)$.
Since $\Spc M_{\Lop,\V \phi} \cap \Spc N_\Lop=\{0\}$, this expresses the structural property that $\Spc M_\Lop(\R^d)=\Spc M_{\Lop,\V \phi}(\R^d) \oplus \Spc N_\Lop$.

\vspace*{0.5ex}{\em 3. Identification of the Underlying Norm}. Any element $p \in  \Spc N_\Lop$ is uniquely characterized by its expansion coefficients $\V \phi(p)$ in the basis $\V p$.
The same holds true for $q={\rm Proj}_{\Spc N_\Lop}\{f\} \in  \Spc N_\Lop$ with $\V \phi(q)=\V \phi(f)$ for any $f\in \Spc M_\Lop(\R^d)$.
Since $\Spc M_{\Lop,\V \phi}(\R^d)$ and $\Spc N_\Lop$ are both Banach spaces,  we can equip their direct sum $\Spc M_\Lop(\R^d)$ with the composite norm $\|f\|_{\Lop, \V \phi}=\|w\|_{\Spc M}+\|\V \phi(f)\|_{2}$, with the guarantee that the Banach-space property is preserved.

For the converse implication, we simply identify $\Spc M_{\Lop,\V \phi}(\R^d)$ as the closed subspace of $\Spc M_{\Lop}(\R^d)$ with the property that
$\|f\|_{\Lop, \V \phi}=\|\Lop f\|_{\Spc M}$.
\end{proof}

The connection with the $\Lop$-spline $s$ of Definition \ref{Def:spline}
is that $s \in \Spc M_{\Lop}(\R^d)$ if and only if the $\ell_1$-norm of its spline weights $\M a=(a_1,\ldots,a_K)$ is finite. Indeed, we have that $\|\Lop s\|_{\Spc M}=\|w_\delta\|_{\Spc M}=\sum_{k=1}^K |a_k|=\|\M a\|_{\ell_1}$, owing to the property that $\|\delta(\cdot-\bx_k)\|_{\Spc M}=1$. 

We note that the choice of gTV is essential here since the simpler (and {\em a priori} only slightly more restrictive) $L_1$-norm regularization $\|\Lop s\|_{L_1}$  would exclude the spline solutions that are of interest to us because $\delta \notin L_1(\R^d)$.
\vspace*{1ex}

\review{
Our final ingredient is the identification of the predual space of $\Spc M_\Lop(\R^d)$, which is denoted by $C_{\Lop}(\R^d)$. 
\begin{theorem}[Predual of native space] 
\label{Theo:Predual}
Let $(\V \phi, \V p)$ be a biorthogonal system of $\Spc N_\Lop\subset L_{\infty,n_0}(\R^d)$ and 
$C_{\Lop, \V p}(\R^d)$ be the image of $C_0(\R^d)$ by $\Lop^\ast: C_0(\R^d) \to C_{\Lop, \V p}(\R^d)$.
Then, $\Spc M_\Lop=\big(C_\Lop(\R^d)\big)'$ where $C_\Lop(\R^d)=C_{\Lop, \V p}(\R^d) \oplus \Spc N'_\Lop$
with $\Spc N'_\Lop={\rm span}\{\phi_n\}_{n=1}^{N_0}$. $C_\Lop(\R^d)$ is a Banach space equipped with the norm
\begin{align}
\label{Eq:Dualnorm}
\|f\|'_{\Lop,\V \phi}=\|\Lop_{\V \phi}^{-1\ast}f\|_\infty+\|\V p(f)\|_2.
\end{align}
\revise{where $\Lop_{\V \phi}^{-1\ast}=\big(\Lop_{\V \phi}^{-1}\big)^\ast$ is the adjoint of $\Lop_{\V \phi}^{-1}$.}
Moreover, there exists a constant $C>0$ such that
\begin{align}
\label{EQ:NormInecL1}
\|f\|'_{\Lop,\V \phi}\le C\|f\|_{L_{1,-n_0}}
\end{align}
for any $f\in L_{1,-n_0}(\R^d)$. 
\end{theorem}
The proof is given in Appendix \ref{App:Predual}.
The direct-sum decomposition in Theorem \ref{Theo:Predual} is achieved by means of the operator ${\rm Proj}_{\Spc N'_\Lop}:  f \mapsto q=\sum_{n=1}^{N_0}
\langle p_n, f \rangle \phi_n$ with $\|q\|=\|\V p(f)\|_2=\|\V p(q)\|_2$, which relies on the biorthogonality of $(\V \phi, \V p)$ to project $\Spc C_{\Lop}(\R^d)$ onto $\Spc N'_\Lop(\R^d)$.
This also means that $C_{\Lop, \V p}(\R^d)$ can be defined as $C_{\Lop, \V p}(\R^d)=\{f \in C_{\Lop}(\R^d): \V p(f)=\V 0 \}$, in direct analogy with the definition of $\Spc M_{\Lop, \V \phi}(\R^d)$ in \eqref{Eq:Mspace}.
}

\section{Proof of Theorems \ref{Theo:L1spline} and \ref{Theo:L1spline2}} 
\label{Sec:ProofgTVTheorems}
\review{
Our technique of proof
will be to first establish the optimality of innovation-type solutions of the form that appear in Definition \ref{Def:spline} for  general linear inverse problems defined on  $\Spc M(\R^d)$ (Theorem \ref{Theo:gFisher}) and to then transfer the result to $\Spc M_\Lop(\R^d)$ with the help of the stable inverse operators specified in Theorem 4.
The first step is achieved by generalizing an earlier result by Fisher and Jerome \cite{Fisher1975}.
}

Let $\Spc H$ be the direct sum of $\Spc M(\R^d)=\big(C_0(\R^d)\big)'$ and a finite-dimensional space $\Spc N$ 
equipped with some norm $\|\cdot\|_{\Spc N}$. The generic element of $\Spc H$ is $f=(w,p)$ with $\|f\|_{\Spc H}=\|w\|_{\Spc M}+\|p\|_{\Spc N}$. 

\setcounter{theorem}{6}
\begin{theorem} [Generalized Fisher-Jerome theorem]
\label{Theo:gFisher}
Let $\V F: \Spc H\to \R^M$ with $M\ge N_0=\dim(\Spc N)$ be a weak*-continuous linear map such that
\begin{align} 
B \|p\|_{\Spc N}  \le&\|\V F(0,p)\|_2 \label{eq:frame}
\end{align}
for some constant $B>0$ and every $p\in \Spc N$. 
Let $\Spc C$ be a convex compact subset of $\R^M$ such that
$\Spc U=\V F^{-1}(\Spc C)=\{(w,p)\in \Spc H: \V F(w,p)\in \Spc C\}$ is nonempty (feasibility hypothesis).
Then,
$$\Spc V=\arg \min_{(w,p) \in \Spc U} \|w\|_{\Spc M} $$
is a nonempty, convex, weak*-compact subset of $\Spc H$ with extremal points $(w_\delta,p)$ of the form
\begin{align}
\label{eq:extreme_w}
w_\delta=\sum_{k=1}^K a_k \delta(\cdot-\bx_k)
\end{align}
with $K\le M$ and $\bx_k\in \R^d$ for $k=1,\ldots,K$, and
$\min_{(w,p)\in \Spc U}\|w\|_{\Spc M}=\sum_{k=1}^K |a_k|$.\\
\end{theorem}
Theorem \ref{Theo:gFisher} is the most technical component of our formulation as it involves the weak* topology. The details of the proof are laid out in Appendix \ref{App:gFisher} together with a precise definition of the underlying concepts. 

\review{
The essence of Theorem \ref{Theo:gFisher} is very similar to Fisher-Jerome's original result \cite[Theorem 1]{Fisher1975}, except for two crucial points:
(i) the fact that they are only considering measures defined over a bounded domain $\Omega \subset \R^d$ (or, by extension, on a compact metric space),
and (ii) the nature of the constraints which, in their case, is limited to coordinatewise inequalities of the form
$z_{1,m} \le [\V F(w,p)]_m \le z_{2,m}$. These differences are substantial enough to justify a new, self-contained proof.
In particular, we believe that our extension for functions defined on $\R^d$ (beyond the compact Hausdorff framework of \cite{Fisher1975}) is essential for covering nonlocal operators such as fractional derivatives, and for deploying Fourier-domain/signal-processing techniques.

}

\review{Our primary constraint for the validity of Theorem \ref{Theo:gFisher} is the existence of the lower bound \eqref{eq:frame}. We now show that this property is implicit in the statement of the hypotheses
of Theorem \ref{Theo:L1spline}.}

\review{\begin{proposition}
\label{Prop:lowerbound}
Let $(\phi_n,p_n)_{n=1}^{N_0}$ be a biorthogonal system of $\Spc N_\Lop \subset \Spc M_\Lop(\R^d)$ such that
$
q=\sum_{n=1}^{N_0} \langle \phi_n, q\rangle p_n \mbox{ for all } q \in \Spc N_\Lop.
$
 Then, Condition 3 in Theorem \ref{Theo:L1spline} is equivalent to the existence of a constant $0<B$ 
 such that
\begin{align}
B \|q\|_{\Spc N_\Lop}  \le\|\V \nu(q)\|_2, 
\quad  \forall q \in \Spc N_\Lop
\end{align}
with $\|q\|^2_{\Spc N_\Lop}=\|\V \phi(q)\|_2^2=\sum_{n=1}^{N_0}  |\langle \phi_n, q\rangle|^2$.
\end{proposition}
}

\review{
While there are softer ways of establishing this equivalence, we have chosen an explicit approach that also serves as background for the proof of Theorem \ref{Theo:L1spline2}.
\begin{proof} Any $q \in\Spc N_\Lop$ has a unique expansion $q=\sum_{n=1}^{N_0}c_n p_n$ with $\M c=\V \phi(q)$ and $\|q\|_{\Spc N_\Lop}=\|\M c\|_2$. The property that $q$ is uniquely determined by its measurements $\M b=\V \nu(q)$ is therefore equivalent to $\M c$ also being the solution of the overdetermined system 
$
\revise{\M P} \M c = \M b
$
with 
\begin{align}
\revise{\M P}=[\V \nu(p_1) \ \cdots \ \V \nu(p_{N_0})].
\end{align}
It is well known that such a system is solvable if and only if $(\revise{\M P}^T \revise{\M P})$ is invertible and that its (least-squares) solution is given by
\begin{align*}
\label{Eq:LSc}
\M c=(\revise{\M P}^T\revise{\M P})^{-1}\revise{\M P}^T\M b.
\end{align*}
This characterization then yields the norm estimate
$$
\|q\|_{\Spc N_\Lop}=\|\M c\|_2 \le \frac{\sigma_{\max}(\revise{\M P})}{\sigma^2_{\min}(\revise{\M P})} \|\V \nu(q)\|_2,
$$
where $\sigma_{\min}(\revise{\M P})=\sigma_{\min}(\revise{\M P}^T)$ and $\sigma_{\max}(\revise{\M P})$ are the minimum and maximum singular values of $\revise{\M P}$, respectively. Finally, the invertibility of $(\revise{\M P}^T\revise{\M P})$ is equivalent to $\sigma^2_{\min}(\revise{\M P})=\revise{\lambda_{\min}(\revise{\M P}^T\revise{\M P})}>0$, while the continuity assumption on $\V \nu$ ensures that 
$\sigma_{\max}(\revise{\M P})<\infty$. \revise{The constant is then given by $B=\sigma^2_{\min}(\revise{\M P})/\sigma_{\max}(\revise{\M P})$}.
\end{proof}
}

\begin{proof}[Proof of Theorem \ref{Theo:L1spline}]\ 
Let $(\V \phi,\V p)$ be a biorthogonal system for $\Spc N_\Lop$. 
Then, by Theorem \ref{Theo:gBeppoLevi}, any function $f \in \Spc M_{\Lop}(\R^d)$ has a unique decomposition as
$f=\Lop_{\V \phi}^{-1}w + p$ with $w=\Lop f \in \Spc M(\R^d)$ and $p \in  \Spc N_{\Lop}$.
This allows us to interpret the measurement process $f \mapsto \V \nu(f)=\langle \V \nu, f\rangle$ as the linear map  $\V F: \Spc H \to \R^M$ such that
\begin{align*}
\langle \V \nu, f\rangle&=\langle \V \nu, \Lop_{\V \phi}^{-1}w\rangle+\langle \V \nu, p\rangle \\
&=\langle \Lop_{\V \phi}^{-1\ast}\V \nu, w\rangle+\langle \V \nu, p\rangle= \V F(w,p).
\end{align*}
\review{We also know from Theorem \ref{Theo:Predual} that $\Lop^{-1\ast}_{\VÊ\phi}$ is an isometry $C_\Lop(\R^d) \to C_0(\R^d)$. 
Hence,
the weak*-continuity of $\V \nu: \Spc M_\Lop(\R^d)\to\R^M$ is equivalent to the weak*-continuity
of $\V F: \Spc H \to \R^M$.
The complementary lower bound is given by Proposition \ref{Prop:lowerbound} as
\begin{align*}
B \|p\|_{\Spc M_{\Lop,\V \phi}}  \le&\|\V \nu(p)\|_2=\|\V F(0,p)\|_2.
\end{align*}
With this new representation}, the constrained minimization problem is equivalent to the one considered in Theorem \ref{Theo:gFisher} \review{with $\Spc N=\Spc N_\Lop$}, which ensures that all extreme points of the solution set are of the form $(p,w_\delta)$ with
$w_\delta=\sum_{k=1}^K a_k \delta(\cdot-\bx_k)$, $K\le M$, and $\bx_k\in \R^d$. Upon application of the (stable) right-inverse operator, this maps into
$s=\Lop_{\V \phi}^{-1}w_\delta + p$, where $p$ is a suitable component that is in the null space of the operator. Finally, we use the explicit kernel formula \eqref{Eq:kernelinverseg} and the procedure outlined at the end of Section \ref{Sec:Splines} to convert this representation into
\eqref{eq:spline}, which removes the artificial dependence upon $\V \phi$.
\end{proof}

\begin{proof}[Proof of Theorem \ref{Theo:L1spline2}]\ 
\review{
From the proof of Proposition \ref{Prop:lowerbound}, we know that the minimal singular value of the cross-product matrix $\revise{\M P}=[\M p_1 \cdots \M p_M]^T$ with $\revise{\M p_m=\V \nu(p_m)}\in \R^{N_0}$
is non-vanishing. The geometric implication is that ${\rm span}\{\M p\}_{m=1}^{M}=\R^{N_0}$. Since the corresponding system is redundant, we can always identify a subset of these row vectors that forms a basis of $\R^{N_0}$. Without loss of generality, we now assume that
this subset is $\{\M p_m\}_{m=1}^{N_0}$ and that the corresponding submatrix $\revise{\M P}_0=[\M p_1 \cdots \M p_{N_0}]^T$ is therefore invertible.
In other words, we have identified a reduced vector of measurement functionals $\V \nu_0=(\nu_1, \dots,\nu_{N_0})$ that is linearly independent with respect to $\Spc N_\Lop$.
This, in turn, allows us to construct the boundary functional $\V \phi_0=\revise{\M P}_0^{-1} \V \nu_0$ that meets the biorthogonal requirement $\V \phi_0(p_n)=\revise{\M P}_0^{-1}\V \nu_0(p_n)=\revise{\M P}_0^{-1}\M p_n=\M e_n$. In effect, this yields a biorthogonal system with the property that $\Spc N'_{\Lop}={\rm span}\{\nu_n\}_{n=1}^{N_0} \subset \Spc M'_\Lop(\R^d)$.
}

\review{
Coming back to our interpolation problem, we define $\M y=(\M y_0, \M y_1)$ with $\M y_0=(y_1,\dots,y_{N_0})\in \R^{N_0}$ and $\M y_1\in \R^{M-N_0}$.
Due to the biorthogonality of $(\V \phi_0,\V p)$, the unique element $q_0 \in \Spc N_\Lop$ such that $\V \nu_0(q_0)=\revise{\M P}_0\V \phi_0(q_0)=\M y_0$ is given by
$$
q_0=\sum_{n=1}^{N_0} b_n p_n
$$
with $\M b=(b_1,\dots, b_{N_0})=\revise{\M P}_0^{-1} \M y_0$. The other ingredient is Theorem \ref{Theo:gBeppoLevi}, which ensures that any 
$f \in \Spc M_\Lop(\R^d)$ has a unique decomposition as $f=\Lop_{\V \phi_0}^{-1}w+q$ with $q \in \Spc N_\Lop$ and $w \in \Spc M(\R^d)$. Now, the crucial property is that the boundary conditions 
$\V \phi_0(\Lop_{\V \phi_0}^{-1}w)=\M 0$ imply that $\V \nu_0(\Lop_{\V \phi_0}^{-1}w)=\M 0$ for all $w \in \Spc M_\Lop(\R^d)$.
This allows us rewrite the solution of our generalized interpolation problem as $f=\Lop_{\V \phi_0}^{-1}w_1 + q_0$, where
$$w_1=\arg \min_{w \in \Spc M(\R^d)} \|w\|_{\Spc M} \ \mbox{ s.t. } \ \V \nu_1(\Lop_{\V \phi_0}^{-1}w)=\M y_1-\V \nu_1(q_0).
$$
The result then follows from the continuity of $\Lop_{\V \phi_0}^{-1}$ and the reduced version of Theorem \ref{Theo:gFisher} with $\Spc N=\{0\}$.
}
\end{proof}

\review{
\section{Further Theoretical and Computational Issues}
\label{Sec:OpenProbs}
The analogy with the finite-dimensional theory of compressed sensing raises the fundamental theoretical question: Is it possible to provide conditions on the measurement operator $\V \nu$ 
such that a perfect recovery is possible for certain classes of signals; in particular, splines with a given number of knots? This is an open topic that calls for further investigation. Because the problem is formulated in the continuum, we suspect that it is much more difficult---if not impossible---to identify conditions that ensure unicity.
}

\review{While the reconstruction problem in Theorem \ref{Theo:L1spline}  is formulated in {\em analysis form} ({\em i.e.}, minimization of $\|\Lop s\|_{\Spc M}$), the interesting outcome is that the solution \eqref{eq:spline} is given in {\em synthesis form}, with the unusual twist that the underlying dictionary $\{\rho_\Lop(\cdot-\V \tau)\}_{\V \tau \in \R^d}$ of basis functions is infinite-dimensional and not even countable. This interpretation suggests a natural discretization which is to select
a finite subset of equally-spaced functions $\{\rho_\Lop(\cdot-\V \tau_n)\}_{n=1}^N$ with $N \gg M$ and to rely on linear programming for $\epsilon=0$, or quadratic programming for $\epsilon>0$, or some other convex optimization technique to numerically solve the underlying $\ell_1$-minimization problem.
We have preliminary evidence that this approach is feasible. In particular, we have considered the generalized interpolation scenario covered by Theorem 2 and  observed that the simplex algorithm performs well in the sense that it always returns a nonuniform $\Lop$-spline with a number of knots $K\le (M-N_0)$. The key theoretical question now is to establish the convergence of such a scheme as the sampling step gets smaller. 
}

\review{Since the space that is spanned by the null-space components of $\Lop$ and the integer shifts of $\rho_\Lop$ is the space of cardinal $\Lop$-splines (see \cite{Unser2005} for the generic case of an ordinary differential operator), one may also consider an alternative discretization that uses the corresponding B-spline basis functions, which are much better conditioned than Green's functions. This would bring us back to a numerical problem that is very similar to \eqref{Eq:penalized}, with the advantage of maintaining a direct control over the discretization error. In the case of a pure denoising problem, another possible option is to run the taut-string algorithm \cite{Mammen1997, Scherzer2005}
or some appropriate variation thereof.
}

\review{At any rate, we believe that the issue of the proper discretization of the reconstruction problem \eqref{eq:l1min} as well as the development of adequate numerical schemes are important research topics on their own right. For the cases where the solution is not unique, Theorem \ref{Theo:L1spline} also suggests a new computational challenge: the design of a minimization algorithm that systematically converges to an extremal point of the problem, the best solution being the spline that exhibits the minimal number of knots $K=\|\M a\|_0$.}



\appendix
\section{Proof of Theorem \ref{Theo:stableinverse}}
\label{App:Stable}

\begin{proof}
First, we establish the sufficiency of the stability condition by considering the signal $f(\bx)=\Op G\{w\}(\bx)=\int_{\R^d} g(\V x,\V y) w(\V y)\dint \V y$, where $w \in \Spc M(\R^d)$, and by constructing the estimate
\begin{align*}
|f(\bx)| (1+\|\bx\|)^{-\alpha_0}&= (1+\|\bx\|)^{-\alpha_0} \left| \int_{\R^d} g(\V x,\V y) w(\V y) \dint \V y \right| \\
&\le  (1+\|\bx\|)^{-\alpha_0} \esssup_{\V y \in \R^d} |g(\V x,\V y)| \; \|w\|_{\Spc M},
\end{align*}
which implies that
$$
\|f\|_{\infty,\alpha_0}=\|\Op G\{w\}\|_{\infty,\alpha_0} \le \left(\esssup_{\bx, \V y \in \R^d} |g(\V x,\V y)|\; (1+\|\bx\|)^{-\alpha_0}\right) \|w\|_{\Spc M}
$$
for all $w \in \Spc M(\R^d)$.
In doing so, we have shown that $$\|\Op G\|\le \esssup_{\bx,\V y \in \R^d}  |g(\V x,\V y)|\; (1+\|\bx\|)^{-\alpha_0} < \infty.$$
To prove necessity, we use the property that $g(\V x,\V y)=\Op G\{\delta(\cdot-\V y)\}(\bx)$, where the shifted Dirac impulse $\delta(\cdot-\V y)$ is included in 
$\Spc M(\R^d)$ with $\|\delta(\cdot-\V y)\|_{\Spc M}=1$. We then observe that, for each $\V y \in \R^d$,
$$
\|\Op G\{\delta(\cdot-\V y)\}\|_{\infty,\alpha_0}=\esssup_{\V x \in \R^d} (1+\|\bx\|)^{-\alpha_0}  |g(\V x,\V y)|.
$$
Moreover, $\Op G$ being bounded, we have that
$$\|\Op G\{\delta(\cdot-\V y)\}\|_{\infty,\alpha_0} \leq \lVert \delta(\cdot-\V y)\, \rVert_{\Spc M}\, \lVert G \rVert = \lVert G \rVert,$$
which means that 
$$\esssup_{\V x \in \R^d} (1+\|\bx\|)^{-\alpha_0}  |g(\V x,\V y)|\le \esssup_{\V x, \V y\in \R^d} (1+\|\bx\|)^{-\alpha_0}  |g(\V x,\V y)| \leq \lVert G\rVert < \infty.
$$
As we already know that the inequality holds in the other direction as well, we obtain $$\|\Op G\|=\esssup_{\bx,\V y \in \R^d}  |g(\V x,\V y)|\; (1+\|\bx\|)^{-\alpha_0},$$
which concludes the proof.
\end{proof}
\section{Proof of Theorem \ref{Theo:inverse}}
\label{App:Inverse}
\begin{proof}
\review{
We first establish the properties of the operator on Schwartz' space of smooth and rapidly-decreasing signals $\Spc S(\R^d)$ to avoid any technical problem related to splitting sums and interchanging integrals. We also rely on Schwartz's kernel theorem which states the equivalence
between the continuous
linear operators $\Op G: \Spc S(\R^d) \to \Spc S'(\R^d)$ and their Schwartz kernels (or general impulse response) $g \in \Spc S'(\R^d \times \R^d)$, meaning that two such operators are identical if and only if their kernels are equal---in the sense of distributions.}

Using the explicit representation of $g_{\V \phi}$ together with $\Lop\{\rho_{\Lop}\}=\delta$ (Dirac distribution) and $\Lop\{p_m\}=0$ for $m=1,\dots, N_0$, we then easily show that
$$
\Lop\Lop_{\V \phi}^{-1}\{\varphi\}=\varphi
$$
for all $\varphi \in \Spc S(\R^d)$.
Next, we invoke the biorthogonality property $\langle\phi_m, p_n \rangle=\delta_{m-n}$ (Kronecker delta) to evaluate the inner product of \eqref{Eq:kernelinverseg} with $\phi_m$
as
\begin{align*}
\langle \phi_m,\Lop_{\V \phi}^{-1}\{ \varphi\} \rangle&=\langle \phi_m,\rho_{\Lop} \ast \varphi  \rangle -  \sum_{n=1}^{N_0}  \langle \phi_m, p_n \rangle \langle q_n,\varphi \rangle\\
&=\langle \phi_m,\rho_{\Lop} \ast \varphi  \rangle - \langle q_m,\varphi \rangle \\
&=\langle \phi_m,\rho_{\Lop} \ast \varphi  \rangle - \langle \phi_m,\rho_\Lop \ast \varphi \rangle=0,
\end{align*}
which proves that the boundary conditions are satisfied.

Let us now consider another operator $\Op L^{-1}$ that is also a right inverse of $\Lop$. Clearly, the result of the action of the two operators can only differ by a component that is in the null space of $\Lop$ so that $(\Op L^{-1}\varphi - \Lop_{\V \phi}^{-1}\varphi)=q \in \Spc N_{\Lop}$. Since $(\V \phi, \V p)$ forms a biorthogonal system, $q$ is uniquely determined by $\V \phi (q)$, which implies that the right-inverse operator that imposes the condition $\V \phi(\Op L^{-1}\varphi)=\M 0$ is unique. 
\\
%

\review{Next, we define $C=\sup_{\bx, \V y \in \R^d} \left(|g_{\V \phi}(\V x,\V y)|\;(1+\|\bx\|)^{-n_0}\right)<\infty$ and extract the generic continuity bound
$$
\|\Lop^{-1}_{\V \phi}\varphi\|_{\infty,n_0} \le C \|\varphi\|_{\Spc M}
$$
from the the proof of Theorem \ref{Theo:stableinverse}.
This allows us to extend the domain of the operator from $\Spc S(\R^d)$ to $\Spc M(\R^d)$, by the Hahn-Banach theorem. Since $\Spc S(\R^d)$ is dense in $\Spc M(\R^d)$, we can do likewise for the right-inverse property and the boundary conditions by invoking the continuity of $\Lop^{-1}_{\V \phi}$ and $\V \phi$.}

\review{Alternatively, one can establish this extension indirectly by identifying a specific Banach space $\Spc M_{\Lop,\V \phi}(\R^d)$ and then by showing that it is the bijective image of $\Spc M(\R^d)$ by $\Lop^{-1}_{\V \phi}$ (see proof of Statement 1 in Theorem \ref{Theo:gBeppoLevi},
which also nicely settles the issue of stability).}
 %
%
\end{proof}

\section{Proof of Theorem \ref{Theo:Predual}}
\label{App:Predual}
\begin{proof}
\review{
First, we prove that $C_{\Lop, \V p}(\R^d)$ is isometrically isomorphic to $C_0(\R^d)$.
For any $\varphi \in C_0(\R^d)$, $\Lop^* \varphi = 0$ implies that $\varphi \in \mathcal{N}_{\Lop^\ast} \cap C_0(\R^d) =\{0\}$ (since the basis functions of the null space do not vanish at infinity); \emph{i.e.}, $\varphi = 0$. 
$\Lop^*$ is therefore injective,  and hence bijective since it is surjective $C_0(\R^d) \to C_{\Lop,\V p}(\R^d)$ by definition.
In particular, this implies that the adjoint $\Lop^{-1*}_{\V \phi}$ of the operator $\Lop^{-1}_{\V \phi}$ defined by \eqref{Eq:kernelinverseg} is the inverse of $\Lop^*$ from $C_{\Lop,\V p}(\R^d)$ to $C_0(\R^d)$. Therefore, $C_{\Lop,\V p}(\R^d)$ inherits the Banach-space structure of $C_0(\R^d)$ for the norm $\lVert \Lop^{-1}_{\V \phi} f \rVert_\infty$.}

\review{If $f \in C_{\Lop,\V p}(\R^d)$,
then $f = \Lop^* \varphi$ with $\varphi \in C_0(\R^d)$ and $\langle f , p \rangle = \langle \varphi , \Lop p \rangle = 0$ for any $p \in \mathcal{N}_\Lop$. The unique element of $\mathcal{N}'_\Lop$ orthogonal to $\mathcal{N}_\Lop$ is $0$ so that the sum $C_{\Lop,\V p}(\R^d)  \oplus  \mathcal{N}'_\Lop$ is direct. $\mathcal{N}_\Lop'$ is a (finite-dimensional) Banach space for the norm $ \|\V p(f)\|_2$, implying that $C_\Lop(\R^d) = C_{\Lop,\V p}(\R^d)  \oplus  \mathcal{N}'_\Lop$ is a Banach space for \eqref{Eq:Dualnorm}.}

\review{Next, we recall that $\Lop^{-1}_{\V \phi}$ is continuous and bijective from $\mathcal{M}_{\Lop, \V \phi} (\R^d)$ to $\mathcal{M}(\R^d)$ (Theorem \ref{Theo:gBeppoLevi}), while we have just shown that its adjoint is continuous and bijective from $C_{\Lop,\V p}(\R^d)$ to $C_0(\R^d)$. 
Knowing that $\big(C_0(\R^d)\big)' = \mathcal{M}(\R^d)$, this implies that $\big(C_{\Lop,\V p}(\R^d)\big)' = \mathcal{M}_{\Lop, \V \phi} (\R^d)$. Finally, we have
\begin{equation*}
	(C_{\Lop}(\R^d))' = ( C_{\Lop,\V p}(\R^d)  \oplus  \mathcal{N}'_\Lop)'  = (C_{\Lop,\V p}(\R^d))' \oplus (\mathcal{N}_\Lop')' =  \mathcal{M}_{\Lop, \V \phi} (\R^d) \oplus \mathcal{N}_\Lop = \mathcal{M}_\Lop(\R^d),
\end{equation*}
as expected.}

\review{
To establish the weighted $L_1$-norm inequality, we first observe that the hypotheses
$f \in L_{1,-n_0}(\R^d)$ and $p_n \in L_{\infty,n_0}(\R^d)$ imply that  
$|\langle f, p_n\rangle| \le \|p_n\|_{\infty,n_0}\, \|f\|_{L_{1,-n_0}}$ (by the H\"older inequality).
Likewise, using the stability bound \eqref{Eq:Stabilitybound}, we get
\begin{align*}
|\Lop_{\V \phi}^{-1\ast}\{f\}(\V y)|&=\left| \int_{\R^d} g_{\V \phi}(\V x, \V y)\; f(\V x) \dint \V x \right|\\
& \le  \int_{\R^d} C_{\V \phi} (1+\|\V x\|)^{n_0} |f(\V x)| \dint \V x = C_{\V \phi} \|f\|_{L_1,-n_1},
\end{align*}
which yields $\|\Lop_{\V \phi}^{-1\ast}\{f\}\|_{\infty} \le C_{\V \phi} \|f\|_{L_1,-n_1}$. The desired result is then obtained from the summation of these individual bounds.}
 \end{proof}

\section{Proof of Theorem \ref{Theo:gFisher}}
\label{App:gFisher}

The proof follows the same steps as the original one of Fisher and Jerome  \cite[Theorem 1]{Fisher1975}. 
Yet, it differs in the assumptions and technicalities ({\em i.e.}, the consideration of the non-compact domain $\R^d$ and the use of explicit bounds). We have done our best to make it self-contained.  

As preparation, we recall that the weak*-topology on $\mathcal{M}(\R^d) = \big(\mathcal{C}_0(\R^d)\big)'$ is the locally convex topology associated with the family of semi-norms $p_\varphi(w) = |\langle w,\varphi\rangle|$ for $\varphi\in C_0(\R^d)$.  In particular, a sequence of elements $w_n \in \mathcal{M}(\R^d)$ converges to $0$ for the weak*-topology if and only if    $\langle w_n,\varphi\rangle \rightarrow 0 $ for every $\varphi\in C_0(\R^d)$.  A subset of $\mathcal{M}(\R^d)$ is said to be weak*-closed (weak*-compact, respectively) if it is closed (compact, respectively) for the weak*-topology. We shall use Propositions \ref{prop:BAM}  and \ref{prop:BAH}, which are consequences of the Banach-Alaoglu theorem and its variations \cite[p.68]{RudinFA}. 

\begin{proposition} \label{prop:BAM} Compactness in the weak*-topology of $\mathcal{M}(\R^d)$.
	\begin{itemize}
	\item  For every $\alpha >0$, the set $\mathcal{B}_\alpha = \{ w\in \mathcal{M}(\R^d): \ \lVert w \rVert_{\Spc M} \leq \alpha \}$ is weak*-compact in $\mathcal{M}(\R^d)$.
	\item If $(w_n)$ is a sequence in $\mathcal{M}(\R^d)$, bounded for the TV-norm, then we can extract a subsequence that converges in $\mathcal{M}(\R^d)$ for the weak*-topology.  
	\end{itemize}
\end{proposition}

\review{The second point of Proposition \ref{prop:BAM} is valid because the space $C_0(\R^d)$ is separable.}
These properties also carry over to the Banach space $\mathcal{H}=\Spc M(\R^d)  \oplus \mathcal{N}=\big (C_0(\R^d) \oplus \mathcal{N}' \big)'$, which is endowed with the corresponding weak*-topology:
A sequence $(w_n,p_n)$ in $\mathcal{H}$ vanishes for the weak*-topology if and only if $w_n$ vanishes for the weak*-topology of $\mathcal{M}(\R^d)$ and $\lVert p_n \rVert_{\Spc N}\to 0$. 

\vspace*{2ex}
\begin{proposition} \label{prop:BAH} Compactness in the weak*-topology of $\mathcal{H}$.
	\begin{itemize}
	\item  For every $\alpha_1, \alpha_2 >0$, the set $\mathcal{B}_{\alpha_1,\alpha_2} = \{ (w, p) \in \mathcal{H}: \ \lVert w \rVert_{\Spc M} \leq \alpha_1, \ \lVert p \rVert_{\mathcal{N}} \leq \alpha_2 \}$ is weak*-compact in $\mathcal{H}$.
	\item If $(w_n,p_n)$ is a sequence in $\mathcal{H}$ such that $\lVert w_n \rVert_{\Spc M} + \lVert p_n \rVert_{\mathcal{N}}$ is bounded, then we can extract a subsequence that converges in $\mathcal{H}$ for the weak*-topology.
	\end{itemize}
\end{proposition}

\begin{proof}[Proof of Theorem \ref{Theo:gFisher}]

The proof is divided in two parts. First, we show that $\Spc V$ is a nonempty, convex, and weak*-compact subspace of $\mathcal{H}$. This allows us to specify $\Spc V$ by means of its extremal points via the Krein-Milman theorem. Second, we show that the extremal points have the announced form. We set $\beta =\inf_{(w,p) \in \mathcal{U}} \lVert w \rVert_{\Spc M}.$

\paragraph{Part I: $\Spc V$ is nonempty, convex, and weak*-compact}
	
Since $\V F$ is weak*-continuous, it is also
continuous $\Spc H \to \R^M$ in the topology of $\Spc H$. Hence, there exists a constant $A>0$ such that
\begin{align}
\label{eq:boundF}
\|\V F(w,p)\|_2\le A(\|w\|_{\Spc M}+ \|p\|_{\Spc N_\Lop}).
\end{align}
Let us consider a sequence $(w_n, p_n)_{n \in \N}$ in $\mathcal{U}$ such that $\|w_n\|_{\Spc M}$ decreases to $\beta$. In particular, $\lVert w_n \rVert_{\Spc M}$ is bounded \review{above} by $\lVert w_0 \rVert_{\Spc M}$. We set $M = \max_{\V x \in \mathcal{C}} \lVert \V x \rVert$. 
Using respectively \eqref{eq:frame}, \eqref{eq:boundF}, and $\lVert \V F(w_n,p_n) \rVert_2 \leq M$ (since $(w_n, p_n) \in \mathcal{U}$), we deduce the inequalities
\begin{align}
	\lVert p_n \rVert_{\mathcal{N}} &\leq   \frac{1}{B} \lVert \V F (0, p_n) \rVert_2 =   \frac{1}{B} \lVert \V F (w_n, p_n) - \V F (w_n,0) \rVert_2  \nonumber\\& \leq  \frac{1}{B} ( \lVert \V F (w_n, p_n) \rVert_2 + \lVert \V F (w_n,0) \rVert_2)   \nonumber \\
					&\leq  \frac{1}{B} (M + A \lVert w_n \rVert_{\Spc M} ) \leq  \frac{1}{B} (M + A \lVert w_0 \rVert_{\Spc M} )\label{eq:boundspn}, 
\end{align}
which shows that $p_n$ is bounded. We can then extract a sequence $(w_{s_n}, p_{s_n})$ from $(w_n, p_n)$ that converges to $(w_{\infty}, p_{\infty}) \in \mathcal{H}$ for the weak*-topology (by Proposition \ref{prop:BAH}). 
\review{Since $\|w_n\|\to\beta$ and $\|w_{s_n}\|$ is a subsequence, it must also converge to $\beta$. }


\review{On the other hand, the set $\mathcal{U} = \V F^{-1} (\mathcal{C})$ is weak*-closed in $\mathcal{H}$, as the preimage of a closed set by a weak*-continuous function $\V F$. Consequently, $(w_\infty, p_\infty)$ is the weak*-limit of a sequence of elements in $\mathcal{U}$. We therefore deduce that $(w_\infty, p_\infty)\in \mathcal{U}$, so that $\lVert w_\infty \rVert_{\Spc M}\geq\beta$. In light of the previous inequality, this yields $\lVert w_\infty \rVert_{\Spc M} = \beta$, which proves that $\Spc V$ is not empty.} \\

	In addition to being weak*-closed, the set $ \mathcal{U} = \V F^{-1}(\mathcal{C})$ is convex because $\mathcal{C}$ is convex and $\V F$ linear. 
Likewise, $\mathcal{V} = \mathcal{U} \bigcap \{(w,p): \ \lVert w \rVert_{\Spc M} \leq \beta \}$ is convex, weak*-closed as the intersection of two sets with the same property. Finally, for $(w,p) \in \mathcal{V}$, we show that $\lVert p\rVert_{\mathcal{N}} \leq \frac{M+A \lVert w \rVert_{\Spc M}}{B} = \frac{M + A \beta}{B} = \gamma$, based on the same inequalities as in \eqref{eq:boundspn}. Therefore, we have
	\begin{equation*}
	\mathcal{V} \subset \{ (w,p) \in \mathcal{H}: \ \lVert w \rVert_{\Spc M} \leq \beta, \ \lVert p \rVert_{\Spc N} \leq \gamma\},
	\end{equation*}
	where the set on the right-hand side is weak*-compact, due to Proposition \ref{prop:BAH}. \review{Since any weak*-closed set included in a weak*-compact set is necessarily weak*-compact, this shows that $\mathcal{V}$ is weak*-compact.} \\
		
We are now in the position to apply the Krein-Milman theorem \cite[p. 75]{RudinFA} to the convex weak*-compact set $ \Spc V \subset \mathcal{H}$, which tells us that ``$ \Spc V$ is the closed convex hull of its extreme points in $\mathcal{H}$ endowed with the weak*-topology''. This leads us to the final part of the proof, which is the characterization of those extreme points.
	
\paragraph{Part II: The extreme points of $ \Spc V$ are of the form \eqref{eq:extreme_w}} 

We shall prove that a necessary condition for $(w,p)$ to be an extreme point of $\Spc V$ is that there are no disjoint Borelian sets $E_1, \dotsc, E_{M+1} \subset \R^d$ such that $w (E_m) \neq 0$ for $m=1,\dotsc,M+1$. The only elements of $\mathcal{M}(\R^d)$ satisfying this condition are precisely those described by \eqref{eq:extreme_w}.  \\

We shall proceed by contradiction and assume that there exist disjoint sets $E_1,\dotsc,$ $E_{M+1}$  such that $w(E_m) \neq 0$ for all $m$. 

We denote the restriction of $w$ to $E_m$ as $w_m=w  \One_{E_m}$. We also define $E = \R^d \backslash \bigcup_m E_m$, and $\bar{w} = w  \One_E$  with $\|w\|_{\Spc M}=\beta$.
For $m=1, \dotsc , M+1$, we set $\V y_m = \V F(w_m,p) \in \mathbb{R}^M$. Since any collection of $(M+1)$ vectors in $\R^{M}$ is linearly dependent, there exists $(c_m)_{1\leq m\leq M+1} \neq \V 0$ such that $\sum_{m=1}^{M+1} c_m \V y_m = \V 0$. 

Let $\mu = \sum_{m=1}^{M+1} c_mw_m \in \mathcal{M}(\R^d)$ and $\epsilon \in (-\epsilon_{\max},\epsilon_{\max})$ with $\epsilon_{\max}=1/ \max_m{|c_m|}$, so that $(1 + \epsilon c_m)>0$ and $(1 - \epsilon c_m) >0$ for all $m$. \review{By construction, we have that
\begin{equation}
	\V F ( \mu, p) = \sum_{m=1}^{M+1} c_m \V F (w_m , p) = \sum_{m=1}^{M+1} c_m \V y_m = \V 0
\end{equation}
and, therefore, that $\V F (w \pm \epsilon \mu,p) = \V F (w, p)$. Hence,
 $$ (w \pm \epsilon \mu ,p)   \in \mathcal{U}.$$
Moreover, $w \pm \epsilon \mu = \bar{w} + \sum_{m=1}^{M+1} (1 \pm \epsilon c_m) w_m$.  Since the measures $\bar{w}, w_1, \dotsc,$ $w_{M+1}$ have disjoint supports and $(1 \pm \epsilon c_m) > 0$, we have}
\begin{align}
	\lVert w \pm \epsilon \mu \rVert_{\Spc M} &= \lVert\bar{w} \rVert_{\Spc M} + \sum_{m=1}^{M+1} (1\pm \epsilon c_m) \lVert w_m \rVert_{\Spc M} \nonumber \\
		&= \lVert\bar{w} \rVert_{\Spc M} + \sum_{m=1}^{M+1} \lVert w_m \rVert_{\Spc M} \pm \epsilon \sum_{m=1}^{M+1} c_m  \lVert w_m \rVert_{\Spc M} \nonumber \\
		&=  \lVert w \rVert_{\Spc M} \pm \epsilon \sum_{m=1}^{M+1} c_m  \lVert w_m \rVert_{\Spc M}  \nonumber \\
		&= \beta \pm \epsilon \sum_{m=1}^{M+1} c_m  \lVert w_m \rVert_{\Spc M}.
\end{align}
If $\sum_{m=1}^{M+1} c_m  \lVert w_m \rVert_{\Spc M} \neq 0$, then 
we either have $\lVert w + \epsilon \mu \rVert_{\Spc M} <\beta$ or $\lVert w  - \epsilon \mu \rVert_{\Spc M} < \beta$, which is impossible since the minimum over $\mathcal{U}$ is $\beta$. Hence, 
$$\sum_{m=1}^{M+1} c_m  \lVert w_m \rVert_{\Spc M} = 0$$
and $\lVert w + \epsilon \mu \rVert_{\Spc M} = \lVert w - \epsilon \mu \rVert_{\Spc M} = \beta$, which translates into $(w + \epsilon \mu, p)$ and $(w - \epsilon \mu, p)$ being included in $\Spc V$. This, in turn, implies that
$(w,p) = \frac{1}{2}(w + \epsilon \mu, p) +  \frac{1}{2}(w - \epsilon \mu, p)$
 is not an extreme point of $\Spc V$.  $\hspace{\stretch{1}}$.
 \end{proof}

\section*{Acknowledgments}
The research was partially supported by the Swiss National Science
Foundation under Grant 200020-162343, the Center for Biomedical Imaging (CIBM) of the Geneva-Lausanne Universities and EPFL, and the European Research Council under Grant 692726 (H2020-ERC Project GlobalBioIm).
The authors are thankful to H. Gupta for helpful discussions. \review{They are also grateful to the three anonymous reviewers for their thoughtful suggestions which have helped improve the manuscript}.


\bibliography{L1bib}

\begin{thebibliography}{10}

\bibitem{Adcock2011}
{\sc B.~Adcock and A.~Hansen}, {\em Generalized sampling and
  infinite-dimensional compressed sensing}, Foundations of Computional
  Mathematics,  (2015), pp.~1--61.

\bibitem{Adcock2013}
{\sc B.~Adcock, A.~C. Hansen, C.~Poon, and B.~Roman}, {\em Breaking the
  coherence barrier: A new theory for compressed sensing}, Forum of
  Mathematics, Sigma,  (to appear).

\bibitem{Beck2009}
{\sc A.~Beck and M.~Teboulle}, {\em Fast gradient-based algorithms for
  constrained total variation image denoising and deblurring problems}, IEEE
  Transactions on Image Processing, 18 (2009), pp.~2419--2434.

\bibitem{Beck2009b}
{\sc A.~Beck and M.~Teboulle}, {\em A fast iterative shrinkage-thresholding
  algorithm for linear inverse problems}, SIAM Journal on Imaging Sciences, 2
  (2009), pp.~183--202.

\bibitem{Bredies2013}
{\sc K.~Bredies and H.~Pikkarainen}, {\em Inverse problems in spaces of
  measures}, ESAIM: Control, Optimisation and Calculus of Variations, 19
  (2013), pp.~190--218.

\bibitem{Bruckstein2009}
{\sc A.~M. Bruckstein, D.~L. Donoho, and M.~Elad}, {\em From sparse solutions
  of systems of equations to sparse modeling of signals and images}, SIAM
  Review, 51 (2009), pp.~34--81.

\bibitem{Candes2008a}
{\sc E.~J. Cand{\`e}s}, {\em The restricted isometry property and its
  implications for compressed sensing}, Comptes Rendus de l'Acad{\'e}mie des
  Sciences, 346 (2008), pp.~589--592.

\bibitem{Candes2013b}
{\sc E.~J. Cand{\`e}s and C.~Fernandez-Granda}, {\em Super-resolution from
  noisy data}, Journal of Fourier Analysis and Applications, 19 (2013),
  pp.~1229--1254.

\bibitem{Candes2007}
{\sc E.~J. Cand{\`e}s and J.~Romberg}, {\em Sparsity and incoherence in
  compressive sampling}, Inverse Problems, 23 (2007), pp.~969--985.

\bibitem{Candes2006}
{\sc E.~J. Cand{\`e}s, J.~K. Romberg, and T.~Tao}, {\em Stable signal recovery
  from incomplete and inaccurate measurements}, Communications on Pure and
  Applied Mathematics, 59 (2006), pp.~1207--1223.

\bibitem{Chambolle2004}
{\sc A.~Chambolle}, {\em An algorithm for total variation minimization and
  applications}, Journal of Mathematical Imaging and Vision, 20 (2004),
  pp.~89--97.

\bibitem{Cohen2003}
{\sc A.~Cohen, W.~Dahmen, I.~Daubechies, and R.~DeVore}, {\em Harmonic analysis
  of the space {BV}}, Revista Matematica Iberoamericana, 19 (2003),
  pp.~235--263.

\bibitem{Dahmen:Micchelli:1987}
{\sc W.~Dahmen and C.~Micchelli}, {\em On theory and application of exponential
  splines}, in Topics in Multivariate Approximation, C.~Chui, L.~Schumaker, and
  F.~Utreras, eds., Academic Press, New York, 1987, pp.~37--46.

\bibitem{Daubechies2004}
{\sc I.~Daubechies, M.~Defrise, and C.~De~Mol}, {\em An iterative thresholding
  algorithm for linear inverse problems with a sparsity constraint},
  Communications on Pure and Applied Mathematics, 57 (2004), pp.~1413--1457.

\bibitem{deBoor1978}
{\sc C.~de~Boor}, {\em A Practical Guide to Splines}, Springer-Verlag, New
  York, 1978.

\bibitem{Denoyelle2015support}
{\sc Q.~Denoyelle, V.~Duval, and G.~Peyr{\'e}}, {\em Support recovery for
  sparse deconvolution of positive measures}, arXiv preprint arXiv:1506.08264,
  (2015).

\bibitem{Deny1954}
{\sc J.~Deny and J.-L. Lions}, {\em Les espaces du type de {B}eppo {L}evi},
  Annales de l'Institut Fourier, 5 (1954), pp.~305--370.

\bibitem{Dey2006}
{\sc N.~Dey, L.~Blanc-F{\'e}raud, C.~Zimmer, P.~Roux, Z.~Kam, J.-C.
  Olivo-Marin, and J.~Zerubia}, {\em {R}ichardson-{L}ucy algorithm with total
  variation regularization for {3D} confocal microscope deconvolution},
  Microscopy Research and Technique, 69 (2006), pp.~260--266.

\bibitem{Donoho2006}
{\sc D.~L. Donoho}, {\em Compressed sensing}, IEEE Transactions on Information
  Theory, 52 (2006), pp.~1289--1306.

\bibitem{Donoho2003}
{\sc D.~L. Donoho and M.~Elad}, {\em Optimally sparse representation in general
  (nonorthogonal) dictionaries via {$\ell_1$} minimization}, Proceedings of the
  National Academy of Sciences, 100 (2003), pp.~2197--2202.

\bibitem{Duchon1977}
{\sc J.~Duchon}, {\em Splines minimizing rotation-invariant semi-norms in
  {S}obolev spaces}, in Constructive Theory of Functions of Several Variables,
  W.~Schempp and K.~Zeller, eds., Springer-Verlag, Berlin, 1977, pp.~85--100.

\bibitem{Duval2015}
{\sc V.~Duval and G.~Peyr{\'e}}, {\em Exact support recovery for sparse spikes
  deconvolution}, Foundations of Computational Mathematics, 15 (2015),
  pp.~1315--1355.

\bibitem{Elad2010b}
{\sc M.~Elad}, {\em Sparse and Redundant Representations. From Theory to
  Applications in Signal and Image Processing}, Springer, 2010.

\bibitem{Elda2015}
{\sc Y.~Eldar}, {\em Sampling Theory: Beyond Bandlimited Systems}, Cambridge
  University Press, 2015.

\bibitem{Fernandez2016}
{\sc C.~Fernandez-Granda}, {\em Super-resolution of point sources via convex
  programming}, Information and Inference, 5 (2016), pp.~251--303.

\bibitem{Figueiredo2003}
{\sc M.~A.~T. Figueiredo and R.~D. Nowak}, {\em An {EM} algorithm for
  wavelet-based image restoration}, IEEE Transactions on Image Processing, 12
  (2003), pp.~906--916.

\bibitem{Fisher1975}
{\sc S.~Fisher and J.~Jerome}, {\em Spline solutions to {$L^1$} extremal
  problems in one and several variables}, Journal of Approximation Theory, 13
  (1975), pp.~73--83.

\bibitem{Foucart2013}
{\sc S.~Foucart and H.~Rauhut}, {\em A Mathematical Introduction to Compressive
  Sensing}, Springer, 2013.

\bibitem{Gelfand-Shilov1964}
{\sc I.~M. Gelfand and G.~Shilov}, {\em Generalized Functions. {V}ol. 1.
  {P}roperties and Operations}, Academic Press, New York, USA, 1964.

\bibitem{Goldstein2009}
{\sc T.~Goldstein and S.~Osher}, {\em The split {B}regman method for
  {$L_1$}-regularized problems}, SIAM Journal on Imaging Sciences, 2 (2009),
  pp.~323--343.

\bibitem{Grafakos2004}
{\sc L.~Grafakos}, {\em Classical and Modern Fourier Analysis}, Prentice Hall,
  2004.

\bibitem{HelOr1998}
{\sc Y.~Hel-Or and P.~C. Teo}, {\em Canonical decomposition of steerable
  functions}, Journal of Mathematical Imaging and Vision, 9 (1998), pp.~83--95.

\bibitem{Hormander1990}
{\sc L.~H{\"o}rmander}, {\em The Analysis of Linear Partial Differential
  Operators. {I}. {D}istribution Theory and {F}ourier Analysis}, vol.~256,
  Springer-Verlag, Berlin, 2nd~ed., 1990.

\bibitem{Kurokawa1988}
{\sc T.~Kurokawa}, {\em Riesz potentials, higher {R}iesz transforms and {B}eppo
  {L}evi spaces}, Hiroshima Math. J, 18 (1988), pp.~541--597.

\bibitem{Lavery2000}
{\sc J.~E. Lavery}, {\em Shape-preserving, multiscale fitting of univariate
  data by cubic {$L_1$} smoothing splines}, Computer Aided Geometric Design, 17
  (2000), pp.~715--727.

\bibitem{Lustig2007}
{\sc M.~Lustig, D.~L. Donoho, and J.~M. Pauly}, {\em Sparse {MRI}: The
  application of compressed sensing for rapid {MR} imaging}, {Magnetic
  Resonance in Medicine}, 58 (2007), pp.~1182--1195.

\bibitem{Madych1990}
{\sc W.~R. Madych and S.~A. Nelson}, {\em Polyharmonic cardinal splines},
  Journal of Approximation Theory, 60 (1990), pp.~141--156.

\bibitem{Mammen1997}
{\sc E.~Mammen and S.~van~de Geer}, {\em Locally adaptive regression splines},
  Annals of Statistics, 25 (1997), pp.~387--413.

\bibitem{Ramani2011}
{\sc S.~Ramani and J.~Fessler}, {\em Parallel {MR} image reconstruction using
  augmented {L}agrangian methods}, IEEE Transactions on Medical Imaging, 30
  (2011), pp.~694 --706.

\bibitem{Rauhut2008}
{\sc H.~Rauhut, K.~Schnass, and P.~Vandergheynst}, {\em Compressed sensing and
  redundant dictionaries}, IEEE Transactions on Information Theory, 54 (2008),
  pp.~2210--2219.

\bibitem{Rudin1992}
{\sc L.~I. Rudin, S.~Osher, and E.~Fatemi}, {\em Nonlinear total variation
  based noise removal algorithms}, Physica D, 60 (1992), pp.~259--268.

\bibitem{Rudin1987}
{\sc W.~Rudin}, {\em Real and Complex Analysis}, McGraw-Hill, New York,
  3rd~ed., 1987.

\bibitem{RudinFA}
{\sc W.~Rudin}, {\em Functional Analysis}, McGraw-Hill, New York, 1991.

\bibitem{Scherzer2005}
{\sc O.~Scherzer}, {\em Taut-string algorithm and regularization programs with
  g-norm data fit}, Journal of Mathematical Imaging and Vision, 23 (2005),
  pp.~135--143.

\bibitem{Schoenberg1946}
{\sc I.~J. Schoenberg}, {\em Contributions to the problem of approximation of
  equidistant data by analytic functions}, Quarterly of Applied Mathematics, 4
  (1946), pp.~45--99, 112--141.

\bibitem{Schultz1967}
{\sc M.~H. Schultz and R.~S. Varga}, {\em L-splines}, Numerische Mathematik, 10
  (1967), pp.~345--369.

\bibitem{Schumaker2007}
{\sc L.~L. Schumaker}, {\em Spline Functions: {B}asic Theory}, Cambridge
  University Pressity Press, Cambridge, 3rd edition~ed., 2007.

\bibitem{steidl2006splines}
{\sc G.~Steidl, S.~Didas, and J.~Neumann}, {\em Splines in higher order {TV}
  regularization}, International Journal of Computer Vision, 70 (2006),
  pp.~241--255.

\bibitem{Tibshirani1996}
{\sc R.~Tibshirani}, {\em Regression shrinkage and selection via the {Lasso}},
  Journal of the Royal Statistical Society. Series B, 58 (1996), pp.~265--288.

\bibitem{Unser2000}
{\sc M.~Unser}, {\em Sampling---50 years after {S}hannon}, Proceedings of the
  IEEE, 88 (2000), pp.~569--587.

\bibitem{Unser2000a}
{\sc M.~Unser and T.~Blu}, {\em Fractional splines and wavelets}, {SIAM}
  Review, 42 (2000), pp.~43--67.

\bibitem{Unser2005}
{\sc M.~Unser and T.~Blu}, {\em Cardinal exponential splines: {P}art
  {I}---{T}heory and filtering algorithms}, IEEE Transactions on Signal
  Processing, 53 (2005), pp.~1425--1449.

\bibitem{Unser2007}
{\sc M.~Unser and T.~Blu}, {\em Self-similarity: {P}art {I}---{S}plines and
  operators}, {IEEE} Transactions on Signal Processing, 55 (2007),
  pp.~1352--1363.

\bibitem{Unser2016}
{\sc M.~Unser, J.~Fageot, and H.~Gupta}, {\em Representer theorems for
  sparsity-promoting {$\ell_1$} regularization}, IEEE Transactions on
  Information Theory, 62 (2016), pp.~5167--5180.

\bibitem{Unser2014book}
{\sc M.~Unser and P.~D. Tafti}, {\em An Introduction to Sparse Stochastic
  Processes}, Cambridge University Press, 2014.

\bibitem{Vetterli2002}
{\sc M.~Vetterli, P.~Marziliano, and T.~Blu}, {\em Sampling signals with finite
  rate of innovation}, {IEEE} Transactions on Signal Processing, 50 (2002),
  pp.~1417--1428.

\bibitem{Ward20l4}
{\sc J.~Ward and M.~Unser}, {\em Approximation properties of {S}obolev splines
  and the construction of compactly supported equivalents}, {SIAM} Journal on
  Mathematical Analysis, 46 (2014), pp.~1843--1858.

\bibitem{Wendland2005}
{\sc H.~Wendland}, {\em Scattered Data Approximations}, Cambridge University
  Press, 2005.

\end{thebibliography}

\end{document}